\documentclass{amsart}
\pdfoutput=1

\newcommand{\Z}{\mathbb{Z}}
\newcommand{\R}{\mathbb{R}}

\renewcommand{\P}{\mathcal{P}}
\newcommand{\M}{\mathscr{M}}
\newcommand{\Graph}{\mathbb{G}}
\newcommand{\las}{\mathop{\mathrm{las}}\nolimits}
\newcommand{\Tr}{\mathop{\mathrm{Tr}}\nolimits}
\newcommand{\pack}{\mathop{\mathrm{pack}}\nolimits}
\newcommand{\cov}{\mathop{\mathrm{cov}}\nolimits}
\newcommand{\supp}{\mathop{\mathrm{supp}}\nolimits}
\newcommand{\topo}{\mathrm{top}}

\usepackage{amssymb}
\usepackage{enumerate}
\usepackage{mathrsfs}
\usepackage{microtype}
\usepackage{tikz-cd}
\usepackage[pdftex,colorlinks,citecolor=black,linkcolor=black,urlcolor=black,bookmarks=false]{hyperref}
\def\MR#1{\href{https://www.ams.org/mathscinet-getitem?mr=#1}{MR#1}}
\def\arXiv#1{arXiv:\href{https://arXiv.org/abs/#1}{#1}}
\usepackage{doi}

\newtheorem{theorem}{Theorem}[section]
\newtheorem{proposition}[theorem]{Proposition}

\newtheorem{corollary}[theorem]{Corollary}
\newtheorem{lemma}[theorem]{Lemma}
\theoremstyle{remark}
\newtheorem{remark}[theorem]{Remark}

\theoremstyle{definition}
\newtheorem{definition}[theorem]{Definition}

\numberwithin{equation}{section}
\numberwithin{figure}{section}
\numberwithin{table}{section}

\title{Sphere packing bounds via rescaling}

\author{Henry Cohn}
\address{Microsoft Research New England\\ One Memorial
  Drive\\ Cambridge, MA 02142\\ USA}
\email{cohn@microsoft.com}

\author{Andrew Salmon}
\address{Department of Mathematics\\ Massachusetts Institute of
  Technology\\ Cambridge, MA 02139\\ USA}
\email{asalmon@mit.edu}

\date{August 24, 2021}

\thanks{Salmon was supported by an internship at Microsoft Research
  New England.}

\begin{document}

\begin{abstract}
We study the relationship between local and global density for sphere
packings, and in particular the convergence of packing densities in
large, compact regions to the Euclidean limit. We axiomatize key
properties of sphere packing bounds by the concept of a packing bound
function, and we study the special case of sandwich functions, which
give a framework for inequalities given by the Lov\'asz sandwich
theorem.  We show that every packing bound function tends to a
Euclidean limit on rectifiable sets, generalizing the work of
Borodachov, Hardin, and Saff.  Linear and semidefinite programming
bounds yield packing bound functions, and we develop a Lasserre
hierarchy that converges to the optimal sphere packing density.
\end{abstract}

\maketitle

\tableofcontents

\section{Introduction}

Packing problems in Euclidean space involve a balance between local
and global behavior: the constraint that the bodies being packed
cannot overlap is purely local, but dense local configurations do not
always extend to dense global packings. For example, the regular
dodecahedron is the smallest possible Voronoi cell for a
three-dimensional sphere packing \cite{HalesMcLaughlin}, but regular
dodecahedra do not tile space, and in fact rhombic dodecahedra yield
the best overall packing density \cite{Hales,Flyspeck}. This tension
between local and global optimality is called \emph{geometrical
frustration} in physics (see, for example, \cite{SadocMosseri}), and
it is not well understood mathematically. Aside from the case of
packing convex, centrally symmetric bodies in at most two dimensions
\cite[Section~25]{FejesToth}, it is not known how to reduce the
general packing problem to considering a bounded number of bodies, and
there is little reason to believe such a reduction is always possible,
especially in high dimensions. Instead, the global packing density can
be obtained as a limit of the packing densities in bounded but
increasingly large regions \cite{Groemer}, and this limiting process
seems to be essential.  For packing of spheres and covering by bodies
of bounded diameter, the first appearance of this sort of limit
appears to be in a paper by Kolmogorov and Tikhomirov
\cite[Theorem~IX]{kolmogorovtikhomirov}, which proves a limit theorem
for Jordan-measurable sets (see \cite{kolmogorovselected} for an
English translation).

Borodachov, Hardin, and Saff \cite{borodachov2007asymptotics} analyzed
this limit in broad generality. To formulate their result, we need
some notation.  Given a bounded subset $C$ of $\R^d$, let $\pack(C)$
be the largest possible size of a subset $X$ of $C$ such that all
points in $X$ are at distance at least $2$ from each other. In other
words, $\pack(C)$ is the maximum number of unit spheres that can be
centered at points of $C$, if their interiors are not allowed to
overlap. Borodachov, Hardin, and Saff work with subsets that may not
be full-dimensional, such as the surface of a sphere in $\R^d$.
Recall that a Borel subset of $\R^d$ is called \emph{$n$-rectifiable}
if it is the image of a bounded Borel subset of $\R^n$ under a
Lipschitz function from $\R^n$ to $\R^d$.  Let $\mathcal{H}_n$ denote
$n$-dimensional Hausdorff measure on $\R^d$, let $\mathcal{L}_d =
\mathcal{H}_d$ denote $d$-dimensional Lebesgue measure, let $B_r^d(x)$
be the open ball of radius $r$ in $\R^d$ centered at $x$, and let $rC$
be $C$ dilated by a factor of~$r$.

The following theorem is a special case of Theorem~2.2 in
\cite{borodachov2007asymptotics}, after rescaling:

\begin{theorem}[Borodachov, Hardin, and Saff
\cite{borodachov2007asymptotics}]\label{theorem:classicalpackingasymptotics}
Let $1 \le n \le d$, and let $C$ be a compact, $n$-rectifiable subset
of $\R^d$ with $\mathcal{H}_n(C) > 0$. Then the limit
\[
\lim_{r \to \infty} \frac{\pack(rC)}{\mathcal{H}_n(rC)}
\mathcal{H}_n(B_1^n)
\]
exists and equals the sphere packing density in $\R^n$.
\end{theorem}

In this paper, we begin by extending
Theorem~\ref{theorem:classicalpackingasymptotics} to a class of
functions we call packing bound functions. The function $\pack$
defined above will be a packing bound function, as will various upper
bounds for $\pack$, including linear and semidefinite programming
bounds. Our extension thus analyzes how these sphere packing bounds
behave when applied to increasingly large regions of space. For
comparison, Hardin, Saff, and Vlasiuk \cite{HSV} analyze conditions on
short-range interactions between particles that suffice to obtain
fairly general asymptotics. Our approach differs conceptually in
studying not just optimization problems over particle configurations,
but also related quantities such as bounds.

To define packing bound functions, we use the following notation. Let
$d(x,y) = |x-y|$ be the standard $\ell^2$ metric on $\R^d$ (the two
uses of $d$ are not ambiguous in practice), and let $\overline{C}$
denote the closure of $C$. For two subsets $C$ and $C'$ of $\R^d$, the
distance $d(C, C')$ is the infimum over all distances $d(x,x')$ such
that $x \in C$ and $x' \in C'$ (with $d(C, C')=\infty$ if $C$ or $C'$
is empty), and $C(\varepsilon) = \{x \in \R^d: d(x,C) < \varepsilon\}$
is the $\varepsilon$-neighborhood of $C$.  We say a function $\psi
\colon C \to C'$ is \emph{distance-increasing} if $d(\psi(x),\psi(y))
\ge d(x,y)$ for all $x,y \in C$. Let $\mathcal{B}_d$ be the set of all
bounded Borel subsets of $\R^d$, and let $\mathcal{B} = \bigcup_d
\mathcal{B}_d$.

\begin{definition}\label{defpackingbound}
A \emph{packing bound function} is a map $A$ from $\mathcal{B}$ to
$[0,\infty)$ such that the following axioms hold for all elements $C$
and $C'$ of $\mathcal{B}$ and $\varepsilon>0$:
\begin{enumerate}
    \item (Sphere bound) If $C$ is a nonempty set contained in the
      interior of a ball of radius $1$, then $A(C) = 1$.
    \item (Lipschitz inequality) If there exists a distance-increasing
      function $\psi \colon C \to C'$, then $A(C) \le A(C')$.
    \item (Union axiom) If $C$ and $C'$ are subsets of the same
      ambient space $\R^d$ and $d(C, C') \ge 2$, then $A(C \cup C') =
      A(C) + A(C')$.
    \item (Mesh axiom) If $C \subseteq C'(\varepsilon)$, then
      $A(\frac{1}{1+\varepsilon}C) \le A(C')$.
\end{enumerate}
\end{definition}

One example is the function $\pack$. For another, let $\cov(C)$ be the
size of the smallest covering of $C$ by open balls of radius $1$, with
sphere centers not necessarily in $C$.  In other words, for a subset
$C$ of $\R^d$, $\cov(C)$ is the infimum of $|X|$ over subsets $X
\subseteq \R^d$ such that $C \subseteq X(1)$. We will show that both
$\pack$ and $\cov$ are packing bound functions, and that they are the
extreme packing bound functions. More precisely, under the partial
ordering $A_1 \le A_2$ defined by $A_1(C) \le A_2(C)$ for all $C$,
every packing bound function $A$ satisfies $\pack \le A \le \cov$. The
inequality $\pack \le A$ explains the name ``packing bound function.''

Between $\pack$ and $\cov$, we construct a range of packing bound
functions, most notably the linear programming bound. Each of these
functions can be viewed in two ways, as an upper bound for packings or
as a lower bound for coverings,\footnote{This duality is less
interesting than it might sound, because any number less than $1$ is a
trivial lower bound for covering, and any number greater than $1$ is a
trivial upper bound for packing.} and each packing bound function has
consistent large-scale limiting behavior:

\begin{theorem}\label{theorem:packingboundlimit}
Let $1 \le n \le d$, let $C$ be a compact, $n$-rectifiable subset of
$\R^d$ with $\mathcal{H}_n(C) > 0$, and let $A$ be any packing bound
function. Then the limit
\[
\lim_{r \to \infty} \frac{A(rC)}{\mathcal{H}_n(rC)}
\]
exists and depends only on $A$ and $n$, rather than the choice of $C$
or $d$.
\end{theorem}

In fact, we will prove a somewhat more general result in
Theorem~\ref{poppypackingbound}, with the same hypotheses as
Theorem~2.2 in \cite{borodachov2007asymptotics}. For the special case
of the function $\cov$, we learned after completing this work that
convergence was independently proved at the same time by Anderson,
Reznikov, Vlasiuk, and White \cite{ARVW}.

Because $\pack$ is the smallest packing bound function, the limit
\[
\lim_{r \to \infty} \frac{A(rC)}{\mathcal{H}_n(rC)}
\mathcal{H}_n(B_1^n)
\]
is always an upper bound for the sphere packing density in $\R^n$. One
example is the linear programming bound of Cohn and Elkies
\cite{cohn2003new}, which we obtain as a limit of a Delsarte linear
program for bounded regions.  Our formulation of this limit appears to
be new, although some of the techniques are related to work of Cohn,
de Courcy-Ireland, and Zhao \cite{cohn2014sphere,cohn2018gaussian}.

Semidefinite programming hierarchies are an important generalization
of linear programming bounds, which give upper bounds for the size of
the largest packing in any compact set by the work of de Laat and
Vallentin \cite{de2015semidefinite}.  Many of the best bounds known
for spherical codes \cite{bachoc2008new} and sphere packing
\cite{cohn2003new,cohn2002new} come from linear and semidefinite
programming relaxations along these lines.  To formulate these
semidefinite programming hierarchies as packing bound functions, we
will define a discrete Lasserre hierarchy. Unlike the topological
Lasserre hierarchy from \cite{de2015semidefinite}, the discrete
hierarchy uses only the graph structure of distances in the space $C
\subseteq \R^d$ and not its topology.  This construction yields a
packing bound function, and thus abstractly, there must exist a
Euclidean limit.  We show that this Euclidean limit agrees with the
limit of the topological approach from \cite{de2015semidefinite}, and
we formulate it as an optimization problem on Euclidean space.  We
also prove that as $t \to \infty$, the $t$-th level of the hierarchy
approaches the optimal density of sphere packing in Euclidean space.

Although $3$-point bounds for spherical codes have been known since
the work of Bachoc and Vallentin \cite{bachoc2008new}, obtaining
$3$-point bounds or other semidefinite programming bounds that refine
the linear programming bound in $\R^d$ has been a longstanding open
problem.  Formulating these refinements was one of the main
motivations of the present work and will be used in subsequent work
with David de Laat \cite{threepointbounds2021} to give new upper
bounds on sphere packing in all dimensions up through $12$ in which
the exact answer is not known (i.e., not $1$, $2$, $3$, or $8$
dimensions).

\section{Packing bound functions}
\label{sec:pbf}

In this section, $A$ will denote an arbitrary packing bound function,
and $C$ and $C'$ will be elements of $\mathcal{B}$.  We begin by
examining some of the consequences of
Definition~\ref{defpackingbound}.  The sphere bound, Lipschitz, and
union axioms are quite natural, while the mesh axiom is a little more
subtle. Intuitively, it says that a packing can be modified slightly
so that the sphere centers are forced to live on a mesh if that mesh
is sufficiently fine.  Later, we will give some consequences of this
axiom to packings when $C'$ is a dense subset as well as packings on a
countable nested union.  These applications use the following special
case of the mesh axiom, which we refer to as the continuity property:

\begin{proposition}[Uniform continuity]\label{proposition:continuity}
For every $\varepsilon>0$, there exists some $\delta > 0$ such that
for all $C$ and $C'$, if $C \subseteq C'(\delta)$, then $A(C) \le
A((1+\varepsilon)C')$.
\end{proposition}

Specifically, it follows from the mesh axiom that we can take $\delta
= \varepsilon/(1+\varepsilon)$, but the specific choice of $\delta$ is
not important in many applications.

Another important special case of the defining properties is
monotonicity, which we obtain by applying the Lipschitz inequality to
the identity function:

\begin{proposition}[Monotonicity] If $C \subseteq C'$, then $A(C) \le A(C')$.
\end{proposition}

It follows that every packing bound function is an upper bound for
$\pack$:

\begin{corollary} \label{cor:packlowerbd}
For every $C$, $\pack(C) \le A(C)$.
\end{corollary}

\begin{proof}
If $X$ is a subset of $C$ such that all points in $X$ are at distance
at least two from each other, then $A(X)=|X|$ by the union and sphere
bound properties, and $A(X) \le A(C)$ by monotonicity.
\end{proof}

Invariance of packing bound functions under isometry also follows
immediately from the Lipschitz inequality:

\begin{proposition}[Isometry invariance]
If $C$ and $C'$ are isometric, then $A(C) = A(C')$.
\end{proposition}

We also observe a useful and basic inequality, which we call the union
bound:

\begin{proposition}[Union bound]
If $C$ and $C'$ are subsets of the same ambient space, then $A(C \cup
C') \le A(C) + A(C')$.
\end{proposition}

\begin{proof}
First, by monotonicity we reduce to the case when $C$ and $C'$ are
disjoint. Now consider the union of copies of $C$ and $C'$ placed very
far apart (further than $2$ plus the diameter of $C \cup C'$), which
we call $D$.  Then the map from $C \cup C'$ to $D$ sending $C$ to the
copy of $C$ in $D$ and $C'$ to the copy of $C'$ in $D$ is
distance-increasing, and therefore $A(C \cup C') \le A(D) = A(C) +
A(C')$ by the Lipschitz and union properties.
\end{proof}

\begin{proposition}[Density] \label{proposition:density}
If $\varepsilon>0$ and $C' \subseteq C$ is dense, then
\[
A(C) \le A((1+\varepsilon)C').
\]
\end{proposition}

\begin{proof}
Every dense subset $C'$ satisfies $C \subseteq C'(\delta)$ for each
$\delta>0$.  By continuity, we can choose $\delta > 0$ so that $C
\subseteq C'(\delta)$ implies $A(C) \le A((1+\varepsilon)C')$.
\end{proof}

Consider the $n$-cube $I^n = [0,1]^n$.  By the union bound, $A(2I^n)
\le 2^n A(I^n)$, and so $A(2^{k}I^n) / 2^{kn}$ is a weakly decreasing
sequence in $k$.  On the other hand, it is bounded below by $1/2^n$
because $A(2^kI^n) \ge \pack(2^kI^n) \ge (1+2^{k-1})^n$ when $k \ge 1$
from using the subset $X = 2^kI^n \cap 2\Z^n$. Therefore, $A(2^{k}I^n)
/ 2^{kn}$ must converge to some positive number.

\begin{definition} \label{def:Euclidean}
For a packing bound function $A$, let $\delta_{A,n} = \lim_{k \to
  \infty} A(2^k I^n) / 2^{kn}$.  We say that a packing bound function
satisfies the \emph{Euclidean bound} if every bounded Borel subset $C
\subseteq \R^n$ satisfies $A(C) \ge \delta_{A,n} \mathcal{L}_n(C)$.
\end{definition}

Because $A \ge \pack$, the quantity $\delta_{A,n}
\mathcal{L}_n(B_1^n)$ is always an upper bound for the sphere packing
density $\delta_{\pack,n} \mathcal{L}_n(B_1^n)$ in $\R^n$. Similarly,
the inequality $A \le \cov$ from Proposition~\ref{prop:covering} will
imply that $\delta_{A,n} \mathcal{L}_n(B_1^n)$ is a lower bound for
the sphere covering density $\delta_{\cov,n} \mathcal{L}_n(B_1^n)$ in
$\R^n$.

\begin{remark} \label{rem:Euclidean}
By Proposition~\ref{proposition:density}, it suffices to prove the
Euclidean bound when $C$ is compact. Specifically, if the Euclidean
bound holds for $\overline{C}$, then for every $\varepsilon>0$,
\[
A(C) \ge
A\mathopen{}\left(\frac{\overline{C}}{1+\varepsilon}\right)\mathclose{}
\ge \delta_{A,n}
\mathcal{L}_n\mathopen{}\left(\frac{\overline{C}}{1+\varepsilon}\right)\mathclose{}
\ge \delta_{A,n} \frac{\mathcal{L}_n(C)}{(1+\varepsilon)^n}.
\]
Because neighborhoods $C(\varepsilon)$ of bounded Borel sets $C$ are
Jordan-measurable,\footnote{To see why, notice that this amounts to
the claim that the boundary of $C(\varepsilon)$ has Lebesgue measure
zero. The function $x \mapsto d(x,C)$ is a Lipschitz function, and so
it is differentiable almost everywhere by Rademacher’s theorem.  Thus,
it suffices to show that the set of points $x$ where $x \mapsto
d(x,C)$ is differentiable and $d(x,C) = \varepsilon$ has measure
zero. For such points, $\lim_{\delta \to 0}
\mathcal{L}_n(C(\varepsilon) \cap
B_\delta^n(x))/\mathcal{L}_n(B_\delta^n(x)) = 1/2$, and the result
follows by the Lebesgue density theorem.} it further suffices to prove
the Euclidean bound for Jordan-measurable sets.  Specifically,
\[
A(C) \ge
A\mathopen{}\left(\frac{C(\varepsilon)}{1+\varepsilon}\right)\mathclose{}
\]
by the mesh axiom, and
\[
\lim_{\varepsilon \to 0} \mathcal{L}_n(C(\varepsilon)) =
\mathcal{L}_n(C)
\]
when $C$ is compact, by Theorem~3.2.39 in \cite{federer2014geometric}.
Moreover, one can show that the Euclidean bound for $C$ is
automatically satisfied if $C$ is capable of tiling Euclidean space
with overlap of measure $0$.
\end{remark}

\subsection{Packing and covering as packing bound functions}

We say a \emph{packing} in $C$ is a subset $X$ of $C$ such that all
points in $X$ are at distance at least $2$ from each other. Recall
that $\pack(C)$ is the size of the largest packing in $C$.

\begin{proposition} \label{prop:pbfcheck}
The function $\pack$ is a packing bound function that satisfies the
Euclidean bound.
\end{proposition}

\begin{proof}
Among the axioms, the sphere bound, Lipschitz inequality, and union
axiom are immediate. For the mesh axiom, suppose $C \subseteq
C'(\varepsilon)$, and let $X$ be a packing in
$\frac{1}{1+\varepsilon}C$. Then $(1+\varepsilon)X$ is a subset of $C$
with all distances distance at least $2 + 2 \varepsilon$. Each $x \in
X$ is within $\varepsilon$ of some point $y(x)$ in $C'$ because $C
\subseteq C'(\varepsilon)$, and so $Y = \{ y(x) : x \in X\}$ is a
packing in $C'$. It now follows that $A(C') \ge |Y|$ by
Corollary~\ref{cor:packlowerbd}, and thus $A(C') \ge
A(\frac{1}{1+\varepsilon}C)$, as desired.

All that remains is to prove the Euclidean bound, for which we use an
averaging argument. Let $C \subseteq \R^n$ be a Borel set with $C
\subseteq [-r,r]^n$, and choose a packing $X_k$ in $2^k I^n$ for each
$k$ such that
\[
\lim_{k \to \infty} \frac{|X_k|}{2^{kn}} = \delta_{\pack,n}.
\]
We will obtain a lower bound for $\pack(C)$ by averaging over
intersections of translates of $C$ with $X_k$. To do so, consider the
cube $R = [-r,2^k+r]^n$ containing $2^k I^n$. Because $C \subseteq
[-r,r]^n$,
\[
\int_R \#(X_k \cap (C+t)) \, dt = \int_{\R^n} \#(X_k \cap (C+t)) \, dt
= |X_k| \mathcal{L}_d(C).
\]
It follows that for some $t \in R$,
\[
\#(X_k \cap (C+t)) \ge \frac{|X_k| \mathcal{L}_d(C)}{\mathcal{L}_d(R)}
= \frac{|X_k| \mathcal{L}_d(C)}{(2r+2^k)^n},
\]
and therefore
\[
\pack(C) \ge \#((X_k-t) \cap C) \ge \frac{|X_k| }{(2r+2^k)^n}
\mathcal{L}_n(C).
\]
Taking the limit as $k \to \infty$ completes the proof.
\end{proof}

Recall that $\cov(C)$ is the smallest covering of $C$ by open balls of
radius $1$, with sphere centers not necessarily in $C$.  That is, for
$C$ in the ambient space $\R^d$, $\cov(C)$ is the infimum of $|X|$
over subsets $X \subseteq \R^d$ such that $C \subseteq X(1)$.  This
infimum is independent of the ambient space chosen: if $C$ is
contained in a proper subspace $\R^{d'}$ of $\R^d$, then the points in
$X$ can be orthogonally projected to $\R^{d'}$. We call such sets $X$
\emph{coverings} of $C$ with open balls of radius $1$.

\begin{proposition} \label{prop:covering}
The function $\cov$ is a packing bound function.  It is the maximal
packing bound function among all packing bound functions, and it does
not satisfy the Euclidean bound.
\end{proposition}

\begin{proof}
First, we check the sphere bound, union axiom, and mesh axiom.  The
function $\cov$ trivially satisfies the sphere bound. For the union
axiom, if $C$ and $C'$ are separated by distance at least $2$, then
there is no open ball of radius $1$ that intersects both $C$ and $C'$,
and so covering $C \cup C'$ is the same as covering $C$ and $C'$
separately.  For the mesh axiom, suppose $C \subseteq
C'(\varepsilon)$, and let $X$ be a covering of $C'$ with open balls of
radius $1$.  Then $(1+\varepsilon)$-balls centered at $X$ cover
$C'(\varepsilon)$, and so these $(1+\varepsilon)$-balls cover $C$.
Therefore $\frac{1}{1+\varepsilon} X$ gives a covering of
$\frac{1}{1+\varepsilon} C$ by balls of radius $1$.

The Lipschitz inequality requires a little more argument. Let $\psi
\colon C \to C'$ be a distance-increasing map, with $C' \subseteq
\R^{d'}$ and $C \subseteq \R^{d}$, and suppose $X'$ is a covering of
$C'$ with open balls of radius $1$. We would like to obtain a covering
of $C$ of the same size. To do so, note that the function $\psi$ is
injective, and its inverse $\psi^{-1}$ on $\psi(C)$ is a Lipschitz
function with Lipschitz constant $1$. By the Kirszbraun theorem, we
can extend $\psi^{-1}$ to a Lipschitz function $\varphi \colon \R^{d'}
\to \R^d$, again with Lipschitz constant $1$. If $x$ is a point in
$C$, then $d(\psi(x),x')<1$ for some $x' \in X'$, from which it
follows that $d(x,\varphi(x'))<1$. Thus, $\{ \varphi(x') : x' \in
X'\}$ is a covering of $C$, as desired.

To see that $\cov$ is the largest packing bound function, suppose $A$
is some other packing bound function, and write $C \subseteq
\bigcup_{i=1}^N B_i$ where the sets $B_i$ are open balls of radius
$1$.  Then $A(C) \le \sum_{i=1}^N A(B_i) = N$ by monotonicity and the
union bound.

The Euclidean bound fails for $d>1$, because $\mathcal{L}_d(B^d_{1})
\delta_{\cov,d} > 1$, while $\cov(B^d_1) = 1$. Note that the
inequality $\mathcal{L}_d(B^d_{1}) \delta_{\cov,d} > 1$ simply says
that the sphere covering density in $\R^d$ is strictly greater than
$1$, which holds because spheres cannot tile space when $d>1$. (If the
covering density were $1$, then a compactness argument would show that
closed unit balls cover space with only measure-zero overlap. However,
a point in space can be on the boundary of at most two such balls,
which does not yield a covering locally.)
\end{proof}

\subsection{Consequences of uniform continuity}

The idea of Proposition~\ref{proposition:continuity} can be restated
as follows.  If $C$ is an arbitrary set, we can choose any
sufficiently fine mesh $C'$ and stipulate that our sphere centers must
live on $C'$ as long as we dilate the mesh by a small amount.  In the
remainder of this section, we give some consequences of this
continuity property.

\begin{proposition}[Nested union]
Let $C_1 \subseteq C_2 \subseteq \dots$ be Borel sets such that
$\bigcup_i C_i$ is bounded.  Then for each $\varepsilon > 0$, there
exists a $j$ such that
\[
A\mathopen{}\left(\bigcup_i C_i\right)\mathclose{} \le
A((1+\varepsilon)C_j).
\]
\end{proposition}

\begin{proof}
Choose $\delta$ as in Proposition~\ref{proposition:continuity}. Then
\[
\overline{\bigcup_i C_i} \subseteq \mathopen{}\left( \bigcup_i
C_i\right)\mathclose{}(\delta) \subseteq \bigcup_j C_j(\delta).
\]
Because $\overline{\bigcup_i C_i}$ is compact, it must be contained in
$C_j(\delta)$ for some $j$, which completes the proof.
\end{proof}

We will use the normalized Hausdorff measure, by which we mean
\[
\mathcal{H}_n(C) = \lim_{\varepsilon \to 0^+}
\mathcal{H}_{n,\varepsilon}(C),
\]
where
\[
\mathcal{H}_{n, \varepsilon}(C) = \frac{\mathcal{L}_n(B_1^n)}{2^n}
\inf \left\{ \sum_{i \in I} \mathrm{diam}(C_i)^n : C \subseteq
\bigcup_{i \in I} C_i \text{ with } \mathrm{diam}(C_i) <
\varepsilon\right\}.
\]

Recall that the \emph{$n$-dimensional Minkowski content}
$\mathcal{M}_n(C)$ of a set $C \subseteq \R^d$ with $d \ge n$ is
defined by
\[
\mathcal{M}_n(C) = \lim_{\varepsilon \to 0+}
\frac{\mathcal{L}_d(C(\varepsilon))}{\mathcal{L}_{d-n}(B_\varepsilon^{d-n})},
\]
when this limit exists (we take $\mathcal{L}_{0}(B_\varepsilon^{0})=1$
when $n=d$). The \emph{upper} or \emph{lower} Minkowski content,
denoted $\overline{\mathcal{M}_n}$ or $\underline{\mathcal{M}_n}$, is
given by the $\limsup$ or $\liminf$, respectively.  For a compact
$n$-rectifiable set $C$, \cite[Theorem 3.2.39]{federer2014geometric}
tells us that $\mathcal{H}_n(C) = \mathcal{M}_n(C)$.

\begin{proposition}
Let $C \subseteq \R^d$ be a bounded Borel set, and suppose
$\overline{\mathcal{M}_n}(C) < \infty$.  Then
\[\limsup_{r \to \infty} \frac{A(rC)}{r^n} < \infty.\]
\end{proposition}

\begin{proof}
Because $\overline{\mathcal{M}_n}(C) < \infty$, there is some constant
$M$ such that for $\delta$ sufficiently small,
\[\mathcal{L}_d(C(\delta)) \le M \delta^{d-n}.\]
Consider the collection $\mathcal{F}$ of balls of radius $\delta$ at
each point of $C$.  By the Vitali covering lemma
\cite[Theorem~2.1]{mattila1995geometric}, there is a subset
$\mathcal{F}'$ consisting of disjoint balls such that $\bigcup_{B \in
  \mathcal{F}'} 5B$ contains $\bigcup_{B \in \mathcal{F}} B =
C(\delta)$.  Because the balls in $\mathcal{F}'$ are disjoint,
\[
|\mathcal{F}'| \le
\frac{\mathcal{L}_d(C(\delta))}{\mathcal{L}_d(B^d_\delta)} \le \frac{M
  \delta^{d-n}}{\mathcal{L}_d(B^d_1) \delta^d} \le
\frac{M'}{\delta^n},
\]
where $M' = M/\mathcal{L}_d(B^d_1)$. Now we let $r =
1/(6\delta)$. Because $\bigcup_{B \in \mathcal{F}'} 5rB$ covers $rC$
and the spheres have radius $r5\delta<1$, monotonicity and the union
and sphere bounds imply that
\[
A(rC) \le \sum_{B \in \mathcal{F}'} A(5rB) \le M'/\delta^n.
\]
In particular,
\[
A(rC)/r^n \le M'/(r\delta)^n = M' 6^n,
\]
and we conclude that $\limsup_{r \to \infty} A(rC)/r^n < \infty$
because this bound holds whenever $r=1/(6\delta)$ with $\delta$
sufficiently small.
\end{proof}

We isolate a geometric lemma in preparation for
Proposition~\ref{minkzero}.

\begin{lemma}\label{geomlemma}
If $C'$ is any compact subset of $C \subseteq \R^d$ with
$\mathcal{M}_n(C)$ and $\mathcal{M}_n(C')$ defined and $\infty >
\mathcal{M}_n(C') > \mathcal{M}_n(C) - \varepsilon$, then there is a
$\delta^*$ such that for all $\delta \le \delta^*$, there is a packing
$\mathcal{B}$ of disjoint balls of radius $\delta/5$ in $C(\delta)
\setminus C'(\delta)$ containing at most $M \varepsilon / \delta^n$
balls for some universal constant $M$ depending on the dimensions $n$
and $d$, such that $\bigcup_{B \in \mathcal{B}} 5B$ is a covering of
$C \setminus C'(\delta)$.
\end{lemma}

\begin{proof}
Since the Minkowski contents of $C$ and $C'$ are within $\varepsilon$,
for sufficiently small $\delta^* \ge \delta > 0$,
\[\mathcal{L}_d(C(\delta) \setminus C'(\delta)) \le \varepsilon\delta^{d-n}
\cdot 2\mathcal{L}_{d-n}(B_1^{d-n}).\]

For every $\delta$, consider the collection of balls $\mathcal{F}$ of
radius exactly $\delta / 5$ at each point of $C \setminus C'(\delta)$.
By the Vitali covering lemma, there are disjoint balls $\mathcal{B}$
of radius $\delta / 5$ such that $5B$ contains the union of the balls
in $\mathcal{F}$.  These balls are a sphere packing and are contained
in $C(\delta) \setminus C'(\delta)$, and therefore
\[
\mathcal{L}_{d}(B_{\delta/5}^{d}) |\mathcal{B}| \le
\mathcal{L}_d(C(\delta) \setminus C'(\delta)) \le
\varepsilon\delta^{d-n} \cdot 2\mathcal{L}_{d-n}(B_1^{d-n}).
\]
It follows that $|\mathcal{B}| \le M \varepsilon / \delta^n$, where
$M$ depends only on $n$ and $d$.
\end{proof}

In the process of working with an arbitrary rectifiable set, we will
need to pass from a set to a subset of the same Minkowski content that
cuts out a bad subset.  We need to show that this removal does not
affect the first order asymptotics of packing.  This result will also
be used in working with compact sets without appealing to the
Euclidean bound.  The following proposition is an analogue of the
``regularity lemma'' in the Riesz energy setting
\cite[Lemma~8.6.9]{borodachov2019discrete}.

\begin{proposition}\label{minkzero}
For each $n$, bounded Borel set $C \subseteq \R^d$ with
$\mathcal{M}_n(C)$ defined, and $k \ge 1$, there is an
$\varepsilon^*>0$ such that for any $\varepsilon \le \varepsilon^*$,
if $C'$ is any compact subset of $C$ with $\mathcal{M}_n(C')$ defined
and $\infty > \mathcal{M}_n(C') > \mathcal{M}_n(C) - \varepsilon$,
then
\[\left(\frac{k}{k+1}\right)^n \limsup_{r \to \infty} \frac{A(rC)}{r^n}
\le \limsup_{r \to \infty} \frac{A(rC')}{r^n} + M\varepsilon k^n \]
for a universal constant $M$ depending on $n$ and $d$.  The same
statement also holds with $\limsup$ replaced by $\liminf$.
\end{proposition}

\begin{proof}
Choose $\delta^*$ and $M$ as in Lemma~\ref{geomlemma}, and let $\delta
= 1 / (k r) \le \delta^*$ for sufficiently large $r$. By
Lemma~\ref{geomlemma}, there is a collection $5\mathcal{B}$ of at most
$M \varepsilon/\delta^n$ open balls of radius $\delta$ that cover $C
\setminus C'(\delta)$.

Because $C \subseteq (C' \cup (C \setminus C'(\delta)))(\delta)$,
scaling by a factor of $r$ shows that
\[
A\mathopen{}\left(\frac{1}{1+r\delta} rC\right)\mathclose{} \le A(rC')
+ A(rC \setminus rC'(r\delta)).
\]
Now since $1/(1+r\delta) = k/(k+1)$,
\[
\limsup_{r \to \infty} \frac{A\mathopen{}\left(\frac{k}{k+1}
  rC\right)\mathclose{}}{r^n} \le \limsup_{r \to \infty}
\frac{A(rC')}{r^n} + \limsup_{r \to \infty} \frac{A(rC \setminus
  rC'(r\delta))}{r^n}.
\]
Finally, because $rC \setminus rC'(r\delta)$ is covered with at most
$M \varepsilon/\delta^n$ open balls of radius $\delta r = 1/k \le 1$,
the last term on the right side can be bounded by $M\varepsilon k^n$,
while the left side is
\[ \left(\frac{k}{k+1}\right)^n\limsup_{r \to \infty} \frac{A(rC)}{r^n}, \]
which is what we wanted to show. For the $\liminf$ version, we instead
use the inequality
\[
\liminf_{r \to \infty} \frac{A\mathopen{}\left(\frac{k}{k+1}
  rC\right)\mathclose{}}{r^n} \le \liminf_{r \to \infty}
\frac{A(rC')}{r^n} + \limsup_{r \to \infty} \frac{A(rC \setminus
  rC'(r\delta))}{r^n}
\]
and conclude in the same way.
\end{proof}

\section{Packing bound functions via graphs}

Consider the category $\Graph$ whose objects are (loopless,
undirected) graphs $G = (V, E)$ with finite chromatic number, and
whose morphisms are graph homomorphisms.  There is a natural operation
of join, where the join of two graphs $G$ and $H$, denoted $G * H$, is
the graph with vertex set the disjoint union of the vertex sets $V(G)$
and $V(H)$ and edge set all edges in $G$ and $H$ together with all
pairs $(x,y)$ with $x \in V(G)$ and $y \in V(H)$.  For simplicity, we
denote the edge containing vertices $x$ and $y$ by $xy$, and the empty
graph by $\emptyset$. The complement $\overline{G}$ of a graph $G$ has
the same vertex set and the complementary edge set.

By passing to complements, we can consider a related category which
shows the connection to packing problems in discrete geometry more
clearly.  This category $\overline{\Graph}$ has as its objects graphs
with finite clique covering number, meaning there is a finite
collection of cliques $C_1, \dots, C_n$ such that $V \subseteq
\bigcup_{i=1}^n C_i$.  Morphisms from $G$ to $G'$ are maps $f \colon V
\to V'$ with the following two properties:
\begin{enumerate}
    \item If $(f(x),f(y)) \in E'$, then $(x,y) \in E$.  That is, the
      preimage of any edge is an edge.
    \item The preimage of each $x' \in V'$ is a clique, possibly
      empty.
\end{enumerate}
There is a natural operation of disjoint union graphs, which we denote
$\sqcup$.  This is \emph{not} a coproduct on the category.  There is
an isomorphism of categories $\Graph \to \overline{\Graph}$ sending a
graph to its complement, which sends disjoint unions to joins and vice
versa.

In geometric terms, we think of graphs in $\overline{\Graph}$ as
follows: the vertices are points in a bounded Borel subset of $\R^d$,
and the edges indicate which points are at distance less than $2$ from
each other. Equivalently, edges in $\Graph$ indicate which pairs of
points are at distance at least $2$, i.e., far enough away that they
do not interact. We will use this category to construct sandwich
functions (named after the Lov\'asz sandwich theorem \cite{knuth}),
which describe the packing bound functions that depend only on the
underlying graph structure, rather than the additional information
supplied by the metric.

A \emph{sandwich function} is a function $\Psi$ from objects of
$\Graph$ to nonnegative real numbers with the following properties:
\begin{enumerate}
    \item (Sphere bound) $\Psi(\emptyset) = 0$ and $\Psi(G) = 1$ if
      $G$ consists of a single vertex.
    \item (Functoriality) If $G \to G'$ is a morphism, then $\Psi(G)
      \le \Psi(G')$.  That is, $\Psi$ is a functor from $\Graph$ to
      the poset of nonnegative real numbers.
    \item (Additivity of join) For all $G$ and $G'$, $\Psi(G * G') =
      \Psi(G) + \Psi(G')$.
\end{enumerate}

\begin{theorem} \label{theorem:pbffromsandwich}
Every sandwich function $\Psi$ gives a packing bound function $A$ on
bounded subsets of Euclidean spaces as follows. Given a bounded Borel
subset $C$ of $\R^n$, let $G(C)$ be the graph whose vertex set is $C$
and whose edge set is pairs $x,y \in C$ such that $| x - y | \ge
2$. Then $A(C) = \Phi(G(C)).$
\end{theorem}

\begin{proof}
To prove that $A$ is a packing bound function, we check the axioms:

\begin{enumerate}
    \item (Sphere bound) If $C$ is empty, then so is $G(C)$, and
      therefore $A(C) = \Phi(G(C)) = 0$.  If $C$ is contained in an
      open ball of radius $1$, then $G(C)$ has no edges and thus
      admits a morphism in $\Graph$ to a single point, so $A(C) =
      \Phi(G(C)) \le 1$.  On the other hand, a single point admits a
      morphism to any nonempty graph by inclusion, and so $\Phi(G(C))
      \ge 1$ if $C \ne \emptyset$.
    \item (Lipschitz) If $f \colon C \to C'$ is a map satisfying
      $d(x,y) \le d(f(x),f(y))$, then it induces a morphism $f \colon
      G(C) \to G(C')$ on graphs.
    \item (Union) If $C$ and $C'$ are separated by distance at least
      $2$, then $G(C \cup C') = G(C) * G(C')$, and
    \[A(C \cup C') = \Phi(G(C) * G(C')) = \Phi(G(C)) + \Phi(G(C')) = A(C) + A(C').\]
    \item (Mesh) If $C \subseteq C'(\varepsilon)$, then there exists a
      map $f \colon C \to C'$ sending each point in $C$ to a point in
      $C'$ such that $|x- f(x)| < \varepsilon$.  We claim the
      composite map $\widetilde{f} \colon \frac{1}{1+\varepsilon} C
      \to C \to C'$ defined by $\widetilde{f}(x) =
      f((1+\varepsilon)x)$ yields a $\Graph$ morphism $\widetilde{f}
      \colon G(\frac{1}{1+\varepsilon} C) \to G(C')$.  To see why,
      note that if $x,y \in \frac{1}{1+\varepsilon} C$ with $|x-y| \ge
      2$, then $|(1+\varepsilon)x - (1+\varepsilon)y| \ge
      2(1+\varepsilon)$ and thus
    \[\begin{split}
    |\widetilde{f}(x) - \widetilde{f}(y)| &= |f((1+\varepsilon)x) -
    f((1+\varepsilon)y)|\\ &\ge |(1+\varepsilon)x -
    (1+\varepsilon)y|\\ & \quad - |f((1+\varepsilon)x) -
    (1+\varepsilon)x| - |f((1+\varepsilon)y) - (1+\varepsilon)y| \\ &
    \ge 2(1+\varepsilon) - \varepsilon - \varepsilon = 2. \qedhere
    \end{split}\]
\end{enumerate}
\end{proof}

\begin{remark}
Many packing bound functions, including packing and the discrete
Lasserre and $k$-point bound hierarchies, come from sandwich functions
on graphs using this theorem.  However, sphere covering does not, as
we will see below.
\end{remark}

\begin{proposition}
The clique number $\omega$ and chromatic number $\chi$ are sandwich
functions, and every sandwich function $\Phi$ satisfies the following
Lov\'asz sandwich theorem:
\[\omega(G) \le \Phi(G) \le \chi(G).\]
\end{proposition}

\begin{proof}
To check that $\omega$ is a sandwich function, note that the image of
a clique under a graph homomorphism remains a clique, while the union
of two cliques in $G_1$ and $G_2$ becomes a clique in their join.  For
chromatic number, the preimage of an independent set under a graph
homomorphism is independent, so an $n$-coloring pulls back to an
$n$-coloring under a graph homomorphism.  Moreover, the only way to
color the join of two graphs is to color each part separately.

Now suppose $\Phi$ is a sandwich function.  Any clique is the iterated
join of one-point graphs, and its inclusion into the whole graph is a
graph homomorphism, thus showing that if $K_n \to G$ is the inclusion
of a clique then $n = \Phi(K_n) \le \Phi(G)$.  By considering all such
cliques, this shows $\omega(G) \le \Phi(G)$.

On the other hand, $n$-colorings of $G$ correspond to graph
homomorphisms $G \to K_n$, with the color sets being the preimages of
the vertices of $K_n$. Thus, if $G$ has an $n$-coloring, then $\Phi(G)
\le \Phi(K_n) = n$.  By considering all possible colorings, this shows
$\Phi(G) \le \chi(G)$.
\end{proof}

\begin{corollary}
The packing bound function $\cov$ does not come from a sandwich
function.
\end{corollary}

\begin{proof}
Let $C$ be the vertices of an equilateral triangle of side length
strictly between $\sqrt{3}$ and $2$.  This set cannot be covered by a
single sphere, but its underlying graph $G(C)$ is independent.  Thus,
$\chi(G(C))$ is strictly smaller than the covering number, which would
be impossible if covering came from a sandwich function.
\end{proof}

Given a sandwich function $\Phi$, we often identify $\Phi$ with the
corresponding packing bound function notationally. For example, we
write $\delta_{\Phi,n}$ in Definition~\ref{def:Euclidean}.

\begin{remark}
The quantity $\delta_{\chi,n} / \mathcal{L}_n(B_1^n)$ appeared in
\cite[Theorem~IX]{kolmogorovtikhomirov} as a limit of their notion of
$\varepsilon$-covering of Jordan-measurable sets by bodies of diameter
at most $2\varepsilon$ (see \cite[Definition~1]{kolmogorovtikhomirov})
as $\varepsilon \to 0^+$.

It would be interesting to compare $\delta_{\chi,n}$, the ``clique
cover'' or ``diameter cover'' density of Euclidean space $\R^n$, with
the density of the best sphere covering $\delta_{\cov,n}$.  Do these
quantities coincide?  This question was posed by Lenz and Heppes and
is open even in the plane \cite[Section~1.3,
  Conjecture~4]{brassmoserpach}.
\end{remark}

The original Lov\'asz sandwich theorem was proved for a graph
invariant $\vartheta$ and used to give an upper bound for Shannon
capacity.  We will see that both of these quantities, as well as some
generalizations, are sandwich functions.

Shannon capacity is defined in terms of the strong graph product.
Recall that the strong graph product $G \cdot H$ is the graph whose
vertex set is $V(G) \times V(H)$ and whose edge set consists of
distinct pairs $(u_1, v_1), (u_2, v_2) \in V(G) \times V(H)$ such that
$u_1 u_2 \in E(G)$ or $u_1 = u_2$ and likewise $v_1 v_2 \in E(G)$ or
$v_1 = v_2$.  Let $G^n$ denote the $n$-fold strong product of $G$ with
itself.  Shannon capacity, denoted $\Theta(G)$, is $\lim_{n \to
  \infty} \alpha(G^n)^{1/n}$, which exists for graphs whose
complements have finite chromatic number; here $\alpha(G)$ denotes the
independence number of $G$. Shannon capacity is not a sandwich
function, because it is not additive under join, as shown by Alon
\cite{alon}.

We now define the Lov\'asz theta number $\vartheta$ and its variants
$\vartheta'$ and $\vartheta^+$, first for finite graphs
\cite[Sections~4.2 and 4.4]{LaurentRendl}, and then in general.  For
finite graphs $G$, consider positive semidefinite $V(G) \times V(G)$
matrices $M$ (i.e., $|V(G)| \times |V(G)|$ matrices indexed by $V(G)$)
such that $\Tr(M) = 1$.

\begin{definition} \label{def:theta}
For $G$ finite, we define
\begin{enumerate}
\item $\vartheta^+(G)$ to be the maximum of $\sum_{i,j} M_{ij}$ over
  all positive semidefinite $V(G) \times V(G)$ matrices $M$ with trace
  $1$ such that $M_{ij} \le 0$ for $ij \in E(G)$,

\item $\vartheta(G)$ to be the maximum of $\sum_{i,j} M_{ij}$ over all
  positive semidefinite $V(G) \times V(G)$ matrices $M$ with trace $1$
  such that $M_{ij} = 0$ for $ij \in E(G)$, and

\item $\vartheta'(G)$ to be the maximum of $\sum_{i,j} M_{ij}$ over
  all positive semidefinite $V(G) \times V(G)$ matrices $M$ with trace
  $1$ such that $M_{ij} = 0$ for $ij \in E(G)$ and $M_{ij} \ge 0$ for
  all $i,j$.
\end{enumerate}
For $G$ infinite, we define $\vartheta(G) = \sup_{H \subseteq G}
\vartheta(H)$, where the supremum ranges over all induced finite
subgraphs of $G$, and similarly for $\vartheta'$ and $\vartheta^+$.
\end{definition}

Because we can extend a matrix by $0$, $\vartheta(G) = \sup_{H
  \subseteq G} \vartheta(H)$ holds for finite graphs $G$, showing that
our definition is internally consistent.  Moreover, there is no harm
in requiring that our matrices $M$ are symmetric, because we can
average $M$ with its transpose. We will assume symmetry as part of the
definition of positive semidefiniteness.

It is known \cite[Equation 69]{LaurentRendl} that
\[\alpha(G) \le \vartheta'(G) \le \vartheta(G) \le \vartheta^+(G) \le \chi(\overline{G}).\]
Moreover, Lov\'asz showed that $\Theta(G) \le \vartheta(G)$.

For later use, it will be helpful to formulate the dual semidefinite
programs defining $\vartheta$, $\vartheta'$, and $\vartheta^+$.  We
will state the duals in terms of positive semidefinite kernels on $G$,
that is, functions $K \colon V(G) \times V(G) \to \R$ such that
$K(x,y)=K(y,x)$ and $\sum_{x,y \in V(G)} w_x w_y K(x,y) \ge 0$ for all
choices of weights $w\colon V(G) \to \R$. This is of course equivalent
to positive semidefinite $V(G) \times V(G)$ matrices, but the language
of kernels will be convenient.

\begin{proposition}\label{discretetheta}
Let $G$ be a finite graph.  Then
\begin{enumerate}
\item $\vartheta'(G)$ is the minimum of $t$ over all positive
  semidefinite kernels $K$ on $G$ such that $K(x,x)=t-1$ for all $x\in
  V(G)$ and $K(x,y) \le -1$ for $xy \not \in E(G)$,

\item $\vartheta(G)$ is the minimum of $t$ over all feasible $K$ and
  $t$ for $\vartheta'(G)$ that additionally satisfy $K(x,y) = -1$ for
  $xy \not \in E(G)$, and

\item $\vartheta^+(G)$ is the minimum of $t$ over all feasible $K$ and
  $t$ for $\vartheta'(G)$ that additionally satisfy $K(x,y) \ge -1$
  for $xy \in E(G)$.
\end{enumerate}
\end{proposition}

For a proof, see \cite[Sections~4.2 and 4.4]{LaurentRendl}.

We now are ready to prove that these give examples of sandwich
functions.

\begin{theorem}\label{thetasandwich}
The functions $G \mapsto \vartheta(\overline{G})$, $G \mapsto
\vartheta'(\overline{G})$, and $G \mapsto \vartheta^+(\overline{G})$
are sandwich functions.
\end{theorem}

To aid with the proof, we have the following general lemma.

\begin{lemma}\label{finitecheck}
Suppose $\Phi(G) = \sup_{H \subseteq G} \Phi(H)$, where the supremum
is over all induced finite subgraphs.  If $\Phi$ satisfies the
sandwich function axioms for all finite graphs, then $\Phi$ satisfies
them for all graphs in $\Graph$.
\end{lemma}

\begin{proof}
For a graph homomorphism $f \colon G \to H$,
\[
\Phi(G) = \sup_{F \subseteq G} \Phi(F) \le \sup_{F \subseteq G}
\Phi(f(F)) \le \sup_{F' \subseteq H} \Phi(F') = \Phi(H),
\]
where the suprema are over induced finite subgraphs $F$ and $F'$ of
$G$ and $H$, respectively.

Every induced finite subgraph $F$ of the join $G * H$ is contained in
the join $F_G * F_H$ of some induced finite subgraphs $F_G \subseteq
G$ and $F_H \subseteq H$.  Thus, $\Phi(F) \le \Phi(F_G) + \Phi(F_H)$
and so $\Phi(G * H) \le \Phi(G) + \Phi(H)$.  Conversely, for any
induced finite subgraphs $F_G \subseteq G$ and $F_H \subseteq H$,
$\Phi(F_G) + \Phi(F_H) = \Phi(F_G * F_H) \le \Phi(G * H)$, so taking
the supremum over such finite subgraphs gives the other inequality.
Therefore $\Phi(G * H) = \Phi(G) + \Phi(H)$.
\end{proof}

\begin{proof}[Proof of Theorem~\ref{thetasandwich}]
By the lemma, we reduce to working only with finite graphs.  First, we
will prove functoriality. Let $f \colon \overline{G} \to \overline{H}$
be a graph homomorphism.  For a $V(G) \times V(G)$ matrix $M$, we
define a $V(H) \times V(H)$ matrix $f_* M$ by
\[(f_* M)_{xy} = \sum_{\substack{i \in f^{-1}(x) \\ j \in f^{-1}(y)}} M_{ij}.\]

We claim that if $M$ is positive semidefinite, then $f_* M$ is
positive semidefinite.  We want to show that for arbitrary weights
$w_x$ for $x \in V(H)$,
\[\sum_{x,y \in V(H)} w_x w_y (f_* M)_{xy} \ge 0.\]
That holds because the sum equals
\[\sum_{i,j \in V(G)} w_{f(i)} w_{f(j)} M_{ij},\]
which is nonnegative by the positive semidefiniteness of $M$.

To prove functoriality for $\vartheta^+$, we check as follows that if
$M$ is feasible for $\vartheta^+(G)$, then $(\Tr(f_* M))^{-1} f_* M$
is feasible for $\vartheta^+(H)$ and $0 < \Tr(f_* M) \le 1$.  The
preimage of any point is a clique in $G$, so $(f_* M)_{xx} \le \sum_{i
  \in f^{-1}(x)} M_{ii}$ and hence $\Tr(f_*M) \le \Tr(M)=1$.  Since
$f_* M$ is nontrivial and positive semidefinite, its trace must be
strictly positive.  If $xy \in E(H)$, and if $f(i) = x$ and $f(j) =
y$, then $ij \in E(G)$, so $M_{ij} \le 0$.  Taking the sum over the
preimages shows that $(f_* M)_{xy} \le 0$.  This shows that the matrix
is indeed feasible.  Now
\[
\begin{split}
\vartheta^+(G) &= \max_{M} \sum_{i,j \in V(G)} M_{ij}\\ &= \max_{M}
\sum_{x,y \in V(H)} (f_*M)_{xy}\\ & \le \max_{M} \sum_{x,y \in V(H)}
(\Tr(f_* M))^{-1} (f_*M)_{xy} \le \vartheta^+(H),
\end{split}
\]
as desired.

The semidefinite program defining $\vartheta(G)$ is slightly more
stringent, requiring $M_{ij} = 0$ rather than just $M_{ij} \le 0$ for
$ij \in E(G)$.  Suppose $M$ is feasible for $\vartheta(G)$.  As a
result $\Tr(f_* M) = 1$, so we can ignore the normalizing factor
$(\Tr(f_* M))^{-1}$.  By the same reasoning as before, $(f_* M)_{xy} =
\sum_{i \in f^{-1}(x),j \in f^{-1}(y)} M_{ij} = 0$ for $xy \in E(H)$,
since the sum over $i$ and $j$ ranges over pairs such that $ij \in
E(G)$.  This shows that $f_* M$ must be feasible for $\vartheta(H)$,
and so this construction shows that $\vartheta(G) \le \vartheta(H)$.

Finally, suppose $M$ is feasible for $\vartheta'(G)$.  This imposes
nonnegativity on its entries, and so $f_* M$ also has nonnegative
entries and is feasible for $\vartheta'(H)$.  This construction
therefore shows that $\vartheta'(G) \le \vartheta'(H)$, and
functoriality is proved for finite graphs in all cases.

What remains is to show the additivity under join, or from the
perspective of the category $\overline{\Graph}$, additivity under
taking disjoint unions of graphs.  We break up this statement into two
inequalities, which we prove separately.

First, we prove subadditivity, i.e., $\Psi(G \sqcup H) \le \Psi(G) +
\Psi(H)$ when $\Psi$ is $\vartheta$, $\vartheta'$, or
$\vartheta^+$. We prove the inequality using the dual formulation in
terms of kernels $K$. Let $K_G$ and $K_H$ be positive semidefinite
kernels that certify the values $\Psi(G)$ and $\Psi(H)$ in
Proposition~\ref{discretetheta}, with $K_G(x,x) = t_G-1$ and $K_H(y,y)
= t_H-1$.

First, we let $1_{S \times T}(x,y)$ be the function which is $1$ for
$(x,y) \in S \times T$ and $0$ otherwise.  We claim that for each
$\alpha > 0$, the kernel
\[
L_{\alpha} = \alpha^2 1_{V(G)\times V(G)} + \alpha^{-2} 1_{V(H)\times
  V(H)} - 1_{V(G)\times V(H)} - 1_{V(H)\times V(G)}
\]
on $G \sqcup H$ is positive semidefinite.  To see why, note that for
weights $w_x$,
\[ 
\begin{split}
&\sum_{x_1, x_2 \in V(G)} w_{x_1} w_{x_2} \alpha^2 + \sum_{y_1, y_2
    \in V(H)} w_{y_1} w_{y_2} \alpha^{-2} - 2 \sum_{\substack{x \in
      V(G) \\ y \in V(H)}} w_x w_y\\ &\qquad = \left( \alpha \sum_{x
    \in V(G)} w_x - \alpha^{-1} \sum_{y \in V(H)} w_y \right)^2 \ge 0.
\end{split}\]

We combine the kernels $K_G$ and $K_H$ by setting
\[
K = (\alpha^{2} + 1) K_G + (\alpha^{-2} + 1) K_H + L_{\alpha},
\]
which is again positive semidefinite on $G \sqcup H$. In particular,
we take $\alpha^2 = t_H/t_G$. Then for $x \in V(G)$,
\[
K(x,x) = (\alpha^2+1)t_G - 1 = t_G+t_H-1,
\]
and for $y \in V(H)$,
\[
K(y,y) = (\alpha^{-2}+1)t_H - 1 = t_G+t_H-1,
\]
so $K(z,z) = t_G+t_H-1$ for all $z$. All that remains is to show that
$K$ is feasible for $\Psi(G \sqcup H)$.

If $K_G$ is feasible for $\vartheta'(G)$ and $K_H$ is feasible for
$\vartheta'(H)$, then $K$ satisfies $K(z_1, z_2) \le -1$ for $z_1 z_2
\not \in E(G) \cup E(H)$ and is thus feasible for $\vartheta'(G \sqcup
H)$.  By direct computation, we can check that if $K_G(x,y) = -1$ for
$xy \not \in E(G)$ and $K_H(x,y) = -1$ for $xy \not \in E(H)$, then
$K(z_1, z_2) = -1$ for $z_1 z_2 \not\in E(G) \cup E(H)$.  This shows
that feasibility for $\vartheta(G)$ and $\vartheta(H)$ implies
feasibility of $K$ for $\vartheta(G \sqcup H)$.  Finally, if $K_G(x,y)
\ge -1$ for $xy \in E(G)$, then $(\alpha^2 + 1) K_G(x,y) \ge -\alpha^2
- 1$, so $K(x,y) \ge -1$ for $x,y \in V(G)$ and $xy \in E(G)$.
Assuming feasibility of $K_H$ for $\vartheta^+(H)$ as well shows that
$K$ is feasible for $\vartheta^+(G \sqcup H)$.

We now prove the other direction of the inequality.  Let $M(G)$ and
$M(H)$ be symmetric positive semidefinite $V(G) \times V(G)$ and $V(H)
\times V(H)$ matrices, respectively, used to establish lower bounds
for $\Psi(G)$ and $\Psi(H)$.  Define a $V(G) \times V(H)$ matrix
$\sigma(G,H)$ so that
\[\sigma(G,H)_{xy} = \left( \sum_{i \in V(G)} M(G)_{ix} \right)
\left( \sum_{j \in V(H)} M(H)_{jy} \right),\]
and similarly define $\sigma(H,G)$ as the $V(H) \times V(G)$ transpose
matrix.  We define the total mass of the matrices $M(G)$ and $M(H)$ by
$\Sigma(G) = \sum_{ij \in V(G) \times V(G)} M(G)_{ij}$ and $\Sigma(H)
= \sum_{ij \in V(H) \times V(H)} M(H)_{ij}$.  These are objectives for
our semidefinite program that we may assume to be positive.  Define a
$(V(G) \sqcup V(H)) \times (V(G) \sqcup V(H))$ matrix $M(G \sqcup H)$
by
\[M(G \sqcup H)_{xy} = \begin{cases} \frac{\Sigma(G) M(G)_{xy}}{\Sigma(G) + \Sigma(H)}
  & \text{for }(x,y) \in V(G) \times V(G), \\
\frac{\Sigma(H) M(H)_{xy}}{\Sigma(G) + \Sigma(H)} & \text{for }(x,y)
\in V(H) \times V(H), \\ \frac{\sigma(G,H)_{xy}}{\Sigma(G) +
  \Sigma(H)} & \text{for }(x,y) \in V(G) \times V(H), \text{ and}
\\ \frac{\sigma(H,G)_{xy}}{\Sigma(G) + \Sigma(H)} & \text{for }(x,y)
\in V(H) \times V(G).
\end{cases}\]
We claim, first, that $M(G \sqcup H)$ is positive semidefinite.  Let
$w_x$ be arbitrary weights for $x \in V(G) \sqcup V(H)$.  Then we want
to show that
\[(\Sigma(G) + \Sigma(H)) \sum_{x,y} w_x w_y M(G \sqcup H)_{xy} \ge 0.\]
The left side is equal to
\[\begin{split}
&\Sigma(G) \sum_{x,y \in V(G)} w_x w_y M(G)_{xy} + \Sigma(H) \sum_{x,y
  \in V(H)} w_x w_y M(H)_{xy}\\ & \quad + 2 \sum_{\substack{i,x \in
    V(G) \\ j,y \in V(H)}} w_x w_y M(G)_{ix} M(H)_{jy}.
\end{split}\]
By Cauchy-Schwarz inequality,
\[
\Sigma(G) \sum_{x,y \in V(G)} w_x w_y M(G)_{xy} \ge \left( \sum_{i,x
  \in V(G)} w_x M(G)_{ix} \right)^2;
\]
specifically, this is the Cauchy-Schwartz inequality $\langle
w,1\rangle^2 \le \langle 1,1 \rangle \langle w,w \rangle$ for the
inner product on $\R^{V(G)}$ defined by $\langle a,b \rangle =
\sum_{x,y} a_x b_y M_{xy}$.  Similarly,
\[
    \Sigma(H) \sum_{x,y \in V(H)} w_x w_y M(H)_{xy} \ge \left(
    \sum_{j,y \in V(H)} w_y M(H)_{jy} \right)^2.
\]
Applying these inequalities shows that the left side is at least
\[\left( \sum_{i,x \in V(G)} w_x M(G)_{ix} + \sum_{j,y \in V(H)} w_y M(H)_{jy} \right)^2,\]
which is nonnegative.

Next, suppose that $M(G)$ and $M(H)$ have trace $1$.  Then $M(G \sqcup
H)$ has trace $1$ by direct computation.  Additionally,
\[\sum_{x,y \in V(G) \sqcup V(H)} M(G \sqcup H)_{xy} = \Sigma(G) + \Sigma(H)\]
by a straightforward computation.

If $xy$ is an edge in the disjoint union, it is an edge in one of $G$
or $H$, and $M(G \sqcup H)_{xy}$ has the same sign as $M(G)_{xy}$ or
$M(H)_{xy}$, depending on whether the edge is in $G$ or $H$.  Thus, if
$M(G)$ and $M(H)$ are feasible for $\vartheta^+(G)$ and
$\vartheta^+(H)$, respectively, then $M(G \sqcup H)$ is feasible for
$\vartheta^+(G \sqcup H)$.  If $M(G)$ and $M(H)$ are feasible for
$\vartheta$, then $M(G \sqcup H)$ is feasible for $\vartheta(G \sqcup
H)$.  Finally, if $M(G)$ and $M(H)$ have nonnegative entries, then
$M(G \sqcup H)$ has nonnegative entries, so $M(G \sqcup H)$ is
feasible for $\vartheta'(G \sqcup H)$ in this case.  Thus, this
construction completes the proof of additivity under disjoint union.
\end{proof}

\section{Asymptotics of packing bound functions on rectifiable sets}

Our main result is a convergence theorem for all packing bound
functions and all $(\mathcal{H}_n, n)$-rectifiable sets $C$ satisfying
$\mathcal{M}_n(C) = \mathcal{H}_n(C)$.  This recovers a statement
similar to best packing on rectifiable sets
\cite{borodachov2007asymptotics}, which is in fact a little stronger
than Theorem~\ref{theorem:packingboundlimit}.  The structure of the
proof will mimic in several places the proof of the Poppy Seed Bagel
Theorem \cite{hardin2005minimal}.  First, we will show the results for
Jordan-measurable subsets of $\R^n$. As in Section~\ref{sec:pbf}, $A$
will denote an arbitrary packing bound function, $C$ and $C'$ will be
elements of $\mathcal{B}$, and $I^n$ will be the unit cube in $\R^n$.

\begin{proposition}
If $C$ is a product of intervals, then \[\lim_{r\to \infty} A(rC) /
r^n = \delta_{A,n} \mathcal{L}_n(C).\]
\end{proposition}

\begin{proof}
This follows from the union bounds together with the following
geometric fact: for each $\varepsilon>0$, there exist $R$ and
$\varepsilon'>0$ such that for all $r > R$, $rC$ contains disjoint
copies of the unit cube occupying at least a $1 - \varepsilon$
fraction of the volume of $rC$ and separated by distance at least
$\varepsilon'$, and is contained in a cover by copies of the unit cube
of total volume fraction at most $1 + \varepsilon$.

Choosing $k$ such that $A(2^k I^n) 2^{-kn}$ is within $\varepsilon$
the limit $\delta_{A,n}$ and such that $2^k \varepsilon' \ge 2$, and
applying both union bounds, we find that
\begin{align*}
(1-\varepsilon)^2 \delta_{A,n} \mathcal{L}_n(C) &\le (1-\varepsilon)
  \frac{A(2^k I^n)}{2^{kn}} \mathcal{L}_n(C) \\ &\le \frac{A(2^k
    rC)}{2^{kn} r^n} \\ &\le (1+\varepsilon) \frac{A(2^k I^n)}{2^{kn}}
  \mathcal{L}_n(C) \\ &\le (1+\varepsilon)^2 \delta_{A,n}
  \mathcal{L}_n(C). \qedhere
\end{align*}
\end{proof}

\begin{proposition}
If $C$ is Jordan-measurable, then $\lim_{r\to \infty} A(rC) / r^n =
\delta_{A,n} \mathcal{L}_n(C)$.
\end{proposition}

\begin{proof}
Let $\{ D_i \}$ be a finite covering of $C$ by small cubes such that
$\sum_i \mathcal{L}_n(D_i) < \mathcal{L}_n(C) + \varepsilon$.  Then
\begin{align*}
\limsup_{r\to\infty} \frac{A(rC)}{r^n} &\le \sum_i \lim_{r \to \infty}
\frac{A(rD_i)}{r^n} \\ &\le \delta_{A,n} (\mathcal{L}_n(C) +
\varepsilon)
\end{align*}

Let $\{ C_i \}$ be a finite collection of disjoint, closed cubes
contained in $C$ such that $\sum_i \mathcal{L}_n(C_i) >
\mathcal{L}_n(C) - \varepsilon$.  Then $\delta := \min_{i,j} d(C_i,
C_j) > 0$.  For $r > 2 / \delta$,
\[
A(rC) \ge A\mathopen{}\left(r \bigcup_i C_i \right)\mathclose{} =
\sum_i A(r C_i).
\]
Dividing by $r^n$ and passing to the limit infimum as $r \to \infty$,
we conclude that
\[\liminf_{r\to\infty} \frac{A(rC)}{r^n} \ge \delta_{A,n}
(\mathcal{L}_n(C) - \varepsilon),\]
which completes the proof.
\end{proof}

Next we prove the result for all compact sets.

\begin{lemma}\label{compactlim}
For every compact set $D$ in $\R^n$,
\[\limsup_{r \to \infty} \frac{A(rD)}{r^n} \le \delta_{A, n}
\mathcal{L}_n(D).\]
\end{lemma}

\begin{proof}
For every $\varepsilon$, there is a Jordan-measurable set $C$
containing $D$ such that $\mathcal{L}_n(C \setminus D) < \varepsilon$
(see Remark~\ref{rem:Euclidean}).  By the Lipschitz inequality, $A(rC)
\ge A(rD)$, and we conclude the result by dividing by $r^n$ and taking
the $\limsup$.
\end{proof}

For the other direction, we first prove a lemma on compact
Jordan-measurable sets:

\begin{lemma}\label{liminfineq}
For each compact Jordan-measurable $B \subseteq \R^n$ and $\varepsilon
> 0$ fixed, there is an $\varepsilon'>0$ such that for any compact
subset $C \subseteq B$ with $\mathcal{L}_n(C) \ge \mathcal{L}_n(B) -
\varepsilon'$,
\[(1 - \varepsilon) \liminf_{r \to \infty} \frac{A(rB)}{r^n} \le
\liminf_{r\to\infty} \frac{A(rC)}{r^n} + \varepsilon.\]
\end{lemma}

\begin{proof}
This is a consequence of Proposition~\ref{minkzero}, using the fact
that Minkowski content is equal to Lebesgue measure for compact
subsets of $\R^n$ (see \cite[Theorem~3.2.39]{federer2014geometric}).
\end{proof}

We will need the Besicovitch covering lemma (as in
\cite[Theorem~8.6.10]{borodachov2019discrete}), which we state as
follows:

\begin{lemma}
Let $\mu$ be a Borel measure on $\R^p$, let $A \subseteq \R^p$ be a
set of finite $\mu$-measure, and let $\mathcal{F}$ be a set of balls
such that for all $x \in A$, the infimum of $r$ over the set of balls
$B(x, r)$ in $\mathcal{F}$ is $0$ (that is, $\mathcal{F}$ contains
balls of arbitrarily small radius centered at all points of $A$).
Then there is a countable subcollection of $\mathcal{F}$ that are
pairwise disjoint and cover $\mu$-almost all $A$.
\end{lemma}

We are now prepared to prove the complementary inequality to
Lemma~\ref{compactlim} for all compact sets.  This argument adapts the
proof of \cite[Theorem 8.6.11]{borodachov2019discrete}.

\begin{theorem}
Let $A$ be a packing bound function.  For any compact set $C$ in
$\R^n$,
\[\lim_{r \to \infty} \frac{A(rC)}{r^n} = \delta_{A,n} \mathcal{L}_n(C).\]
\end{theorem}

\begin{proof}
Let $\varepsilon > 0$ be fixed.  It suffices to prove that
\[\liminf_{r \to \infty} \frac{A(rC)}{r^n} \ge (1-\varepsilon)^2
\delta_{A,n} \mathcal{L}_n(C) - \varepsilon.\]

Define the set
\[C^* := \left\{ x \in C : \lim_{r \to 0^+}
\frac{\mathcal{L}_n(B(x, r) \cap C)}{\mathcal{L}_n(B(x, r))} = 1 \right\}.\]
This set satisfies $\mathcal{L}_n(C^*) = \mathcal{L}_n(C)$ by the
Lebesgue density theorem.

Now to apply Besicovitch covering lemma, we let the set $\mathcal{F}$
consist of all closed balls $B(x,r)$ around points $x \in C^*$ such
that $r < 1$ and
\[\frac{\mathcal{L}_n(B(x, r) \cap C^*)}{\mathcal{L}_n(B(x, r))} \ge
1 - \frac{\varepsilon'}{\mathcal{L}_n(C(1))},\]
where $\varepsilon' > 0$ can be taken to be arbitrarily small.  By the
Besicovitch covering lemma, we can choose a countable disjoint
subcollection $B_i$ of closed balls whose union covers almost all of
$C^*$ and hence almost all of $C$.  Define $C_i = C \cap B_i$.

We can choose $N$ so that
\[
\mathcal{L}_n\mathopen{}\left(\bigcup_{i=1}^N B_i \right)\mathclose{}
\ge (1-\varepsilon)\mathcal{L}_n(C).
\]
Because $B_i$ is in $\mathcal{F}$ and $\bigcup_i B_i \subseteq C(1)$,
\[ \mathcal{L}_n\mathopen{}\left(\bigcup_{i=1}^N C_i \right)\mathclose{} \ge
\left(1-\frac{\varepsilon'}{\mathcal{L}_n(C(1))}\right)
\mathcal{L}_n\mathopen{}\left(\bigcup_{i=1}^N B_i \right)\mathclose{} \ge
\mathcal{L}_n\mathopen{}\left(\bigcup_{i = 1}^N B_i \right)\mathclose{} - \varepsilon'. \]
Then since the balls are compact sets, they are separated from each
other, and we may find a $\delta>0$ such that $d(B_i, B_j) \ge \delta$
for all distinct $i,j \le N$.  Now Lemma~\ref{liminfineq} tells us
that if $\varepsilon'$ is small enough relative to $\varepsilon$, then
\[
\liminf_{r \to \infty} \frac{A\mathopen{}\left( r \bigcup_{i=1}^N C_i
  \right) \mathclose{}}{r^n} \ge (1 - \varepsilon) \liminf_{r
  \to\infty} \frac{A\mathopen{}\left( r \bigcup_{i=1}^N B_i \right)
  \mathclose{}}{r^n} - \varepsilon.
\]

For $r \ge 2\delta^{-1}$, we can apply the union axiom in the
definition of a packing bound function to obtain
\[
A\mathopen{}\left( r \bigcup_{i=1}^N B_i \right) \mathclose{} =
\sum_{i=1}^N A(rB_i).
\]
Combining these inequalities yields
\[
\begin{split}
\liminf_{r \to \infty} \frac{A(rC)}{r^n} &\ge \liminf_{r \to \infty}
\frac{A\mathopen{}\left( r \bigcup_{i=1}^N C_i \right)
  \mathclose{}}{r^n}\\ &\ge (1 - \varepsilon) \liminf_{r \to\infty}
\frac{A\mathopen{}\left( r \bigcup_{i=1}^N B_i \right)
  \mathclose{}}{r^n} - \varepsilon\\ &= (1 - \varepsilon) \liminf_{r
  \to\infty} \sum_{i=1}^N \frac{A(rB_i)}{r^n} - \varepsilon\\ &=
(1-\varepsilon) \sum_{i=1}^N \delta_{A, n} \mathcal{L}_n(rB_i) -
\varepsilon,
\end{split}
\]
where the last equality holds because $B_i$ is
Jordan-measurable. Finally, we obtain a lower bound of
\[
(1-\varepsilon) \sum_{i=1}^N \delta_{A, n} \mathcal{L}_n(rB_i) -
\varepsilon \ge (1-\varepsilon)^2 \delta_{A,n} \mathcal{L}_n(C) -
\varepsilon,\] as desired.
\end{proof}

\begin{remark}
The above theorem can be proved much more easily if the Euclidean
bound is available for $A$.
\end{remark}

Proposition~\ref{proposition:density} automatically gives an extension
to arbitrary subsets of $\R^n$ that uses Minkowski content instead of
Lebesgue measure.

\begin{corollary}
For an arbitrary bounded Borel subset $C \subseteq \R^n$,
\[\lim_{r \to \infty} \frac{A(rC)}{r^n} = \delta_{A,n} \mathcal{M}_n(C).\]
\end{corollary}

Note that because $C \subseteq \R^n$, the Minkowski content
$\mathcal{M}_n(C)$ always exists.

\begin{proof}
Let $\overline{C}$ be the closure of $C$, which is compact and
$n$-rectifiable and hence satisfies $\mathcal{M}_n(C) =
\mathcal{M}_n(\overline{C}) = \mathcal{L}_n(\overline{C})$.  By
density, this implies for any $\varepsilon > 0$,
\[ A(r\overline{C}) \le A(r(1+\varepsilon)C) \le A(r(1+\varepsilon)\overline{C}) \]
and after dividing by $r^n$ the left and right sides both converge to
$\delta_{A,n} \mathcal{L}_n(\overline{C}) (1 + O(\varepsilon))$ as $r
\to \infty$.
\end{proof}

Now, we wish to extend the result to $n$-rectifiable sets, in fact, to
a slightly more general setting.  We need some notions from geometric
measure theory.

\begin{definition}
For a measure $\mu$, a $(\mu, n)$-rectifiable set is a bounded Borel
subset $E$ of $\R^d$ such that there are Lipschitz maps $\psi_i \colon
\R^n \to \R^d$ and bounded Borel subsets $E_i$ of $\R^n$ for which
$\mu(E \setminus \bigcup_i \psi_i(E_i)) = 0$.
\end{definition}

We have the following lemma, which is
\cite[Lemma~3.2.18]{federer2014geometric}:

\begin{lemma}\label{epsiloncharts}
Let $C$ be an $(\mathcal{H}_n, n)$-rectifiable set.  Then for every
$\varepsilon > 0$, there are compact subsets $C_1, C_2, \ldots
\subseteq \R^n$ and bi-Lipschitz maps $\psi_i \colon C_i \to C$ with
Lipschitz constant $1 + \varepsilon$ (in both directions) such that
the sets $\psi_i(C_i)$ are disjoint and
\[\mathcal{H}_n\mathopen{}\left(C \setminus \bigcup_{i=1}^\infty
\psi_i(C_i)\right)\mathclose{} = 0.\]
\end{lemma}

We will also need \cite[Lemma~8.7.2]{borodachov2019discrete}:

\begin{lemma} \label{minkeq}
If $C$ is a compact $(\mathcal{H}_n, n)$-rectifiable set with
$\mathcal{M}_n(C) = \mathcal{H}_n(C)$, then every compact subset $K$
of $C$ is $(\mathcal{H}_n,n)$-rectifiable and satisfies
$\mathcal{M}_n(K) = \mathcal{H}_n(K)$.
\end{lemma}

Now we can prove our main theorem:

\begin{theorem}\label{poppypackingbound}
Let $C \subseteq \R^d$ be an $(\mathcal{H}_n, n)$-rectifiable set with
closure $\overline{C}$ satisfying the property that $\mathcal{M}_n(C)
= \mathcal{H}_n(\overline{C}) < \infty$, and let $A$ be a packing
bound function.  Then
\[\lim_{r \to \infty} \frac{A(rC)}{r^n} = \delta_{A,n} \mathcal{M}_n(C).\]
\end{theorem}

In particular, the above theorem holds for all compact smooth
$n$-manifolds or compact subsets of smooth $n$-manifolds embedded in
$\R^d$ for some $d$.  Since compact $n$-rectifiable sets are also
$(\mathcal{H}_n, n)$-rectifiable and satisfy $\mathcal{M}_n(C) =
\mathcal{H}_n(C)$, Theorem~\ref{poppypackingbound} implies
Theorem~\ref{theorem:packingboundlimit}.

\begin{proof}
By using the density property of $A$, we see that $A(r\overline{C})
\le A(r(1+\varepsilon)C) \le A(r(1+\varepsilon)\overline{C})$, which
reduces the goal to proving the statement for $\overline{C}$.  From
now on we assume $C$ is compact.

We first prove the $\ge$ direction.  In the notation of
Lemma~\ref{epsiloncharts}, choose $N$ so that $\sum_{i=1}^N
\mathcal{H}_n(\psi_i(C_i)) \ge (1-\varepsilon) \mathcal{H}_n(C)$.  By
compactness, there is a $\delta > 0$ such that $d(\psi_i(C_i),
\psi_j(C_j)) \ge \delta$ for all distinct $i,j \le N$.  Then for $r
\ge 2 \delta^{-1}$,
\[ A(rC) \ge \sum_{i=1}^N A(r \psi_i(C_i)) \ge \sum_{i=1}^N
A((1+\varepsilon)^{-1} r C_i ), \]
where the last inequality follows from the Lipschitz property. Thus,
\[
\frac{A(rC)}{r^n\mathcal{H}_n(C)} \ge (1-\varepsilon)
\frac{\sum_{i=1}^N A((1+\varepsilon)^{-1} r C_i )}{\sum_{i=1}^N
  r^n\mathcal{H}_n(\psi_i(C_i))}.
\]
Furthermore, $\mathcal{H}_n(\psi_i(C_i)) \le (1+\varepsilon)^n
\mathcal{H}_n(C_i)$ because $\psi_i$ has Lipschitz constant
$1+\varepsilon$ (this bound follows directly from the definition of
$\mathcal{H}_n$; see, for example,
\cite[Proposition~2.49(iv)]{ambrosiofuscopallara}), and thus
$\mathcal{H}_n(\psi_i(C_i)) \le (1+\varepsilon)^{2n}
\mathcal{H}_n((1+\varepsilon)^{-1}C_i)$.  We conclude that
\[
\liminf_{r \to \infty} \frac{A(rC)}{r^n\mathcal{H}_n(C)} \ge
(1-\varepsilon) (1+\varepsilon)^{-2n} \delta_{A,n},
\]
and the conclusion follows by letting $\varepsilon$ tend to $0$.

Now we wish to prove the other direction,
\[\limsup_{r \to \infty} A(rC) r^{-n} \le \delta_{A,n} \mathcal{M}_n(C).
\]
We note that $\overline{\bigcup_i \psi_i(C_i)}$ is a compact subset of
$C$, and since $\mathcal{M}_n(C) = \mathcal{H}_n(C)$ and the
complement of $\overline{\bigcup_i \psi_i(C_i)}$ has measure $0$ under
$\mathcal{H}_n$, it follows that $\mathcal{M}_n(C) =
\mathcal{H}_n(\overline{\bigcup_i \psi_i(C_i)}) =
\mathcal{M}_n(\overline{\bigcup_i \psi_i(C_i)})$ by
Lemma~\ref{minkeq}.  So by Proposition~\ref{minkzero},
\[\limsup_{r \to \infty} \frac{A(rC)}{r^n} = \limsup_{r\to\infty}
\frac{A\mathopen{}\left(r\,\overline{\bigcup_i \psi_i(C_i)}\right)\mathclose{}}{r^n},\]
and it suffices to prove the corresponding bound for
$\overline{\bigcup_i \psi_i(C_i)}$.

Now by density,
\[ A\mathopen{}\left(r\overline{\bigcup_i \psi_i(C_i)}\right)\mathclose{} \le
A\mathopen{}\left(r(1+\varepsilon) \bigcup_{i=1}^\infty \psi_i(C_i)\right)\mathclose{}. \]
By the nested union and Lipschitz properties,
\begin{align*}
A\mathopen{}\left( r(1+\varepsilon) \bigcup_{i=1}^\infty \psi_i(C_i)
\right)\mathclose{} &\le \lim_{N\to\infty}
A\mathopen{}\left(r(1+\varepsilon)^2 \bigcup_{i=1}^N
\psi_i(C_i)\right)\mathclose{} \\ &\le \sum_{i=1}^\infty
A(r(1+\varepsilon)^3 C_i).
\end{align*}
Dividing by $r^n$ and taking the limit as $r \to \infty$, we find that
\begin{align*}
\limsup_{r \to \infty} \frac{A(rC)}{r^n} &\le (1+\varepsilon)^{3n}
\delta_{A,n} \sum_{i=1}^\infty \mathcal{H}_n(C_i) \\ &\le
(1+\varepsilon)^{4n} \delta_{A,n} \sum_{i=1}^\infty
\mathcal{H}_n(\psi_i(C_i))\\ &= (1+\varepsilon)^{4n} \delta_{A,n}
\mathcal{H}_n(C).
\end{align*}
Letting $\varepsilon \to 0$ completes the result.
\end{proof}

\begin{remark}
We do not know whether Theorem~\ref{poppypackingbound} generalizes to
sets $C$ with non-integral Hausdorff dimension $n$. The limit of
$A(rC)/r^n$ at least makes sense, and the value $\delta_{A,n}$ can
also sometimes be generalized for non-integral $n$. For the linear
programming bound, for example, the radial Fourier transform is used,
which can be written in terms of a Bessel function $J_{\nu}$ with
$\nu=d/2-1$.  By taking $\nu$ to be any real number greater than
$-1/2$, we can formally write down a linear program for non-integral
dimensional Euclidean space, but it is not clear whether it has any
relationship to the limit of $A(rC)/r^n$.  If there is indeed a
generalization in this direction, that would give geometric meaning to
the versions of the linear program when the dimension is not integral.
Could this program even be sharp in some non-integral dimensions?
\end{remark}

\section{The Euclidean limits of $\vartheta'$, $\vartheta$, and $\vartheta^+$}

In this section, we consider the Euclidean limits of the packing bound
functions corresponding to the sandwich functions $\vartheta'$,
$\vartheta$, and $\vartheta^+$.  It turns out that
$\delta_{\vartheta',n} \mathcal{L}_n(B_1^n)$ is the Cohn-Elkies linear
programming bound \cite{cohn2003new} for the sphere packing density in
$\R^n$, which is the best bound known for large $n$. In addition, the
Euclidean limits of $\vartheta$ and $\vartheta^+$ can be computed
exactly, and they are equal to $\mathcal{L}_n(B_1^n)^{-1}$, where
$\mathcal{L}_n(B_1^n)$ is the volume of an $n$-ball of radius $1$.  In
other words,
\[
\delta_{\vartheta,n} = \delta_{\vartheta^+,n} =
\pi^{-\frac{n}{2}}\Gamma\mathopen{}\left(\frac{n}{2} +
1\right)\mathclose{},
\]
and the resulting density bounds for the sphere packing problem are
trivial: $\delta_{\vartheta,n} \mathcal{L}_n(B_1^n) =
\delta_{\vartheta^+,n} \mathcal{L}_n(B_1^n) = 1$.

Despite the weakness of the corresponding sphere packing bounds, the
Euclidean limits $\delta_{\vartheta,n}$ and $\delta_{\vartheta^+,n}$
involve interesting mathematics. These ideas originated in Siegel's
proof of Minkowski's theorem via Poisson summation (see, for example,
Section~2.11.4 in~\cite{dymmckean}).  Minkowski's theorem can be
interpreted as saying that the maximum lattice packing density of a
convex, origin-symmetric body is~$1$, which Siegel proved by applying
Poisson summation to the convolution of the indicator function of the
convex body with itself. It is natural to ask whether another
auxiliary function could prove a better bound than~$1$. Siegel showed
that the answer is no under certain hypotheses \cite{siegel}, which
amounts to computing the limits $\delta_{\vartheta,n}$ and
$\delta_{\vartheta^+,n}$, and Gorbachev \cite{gorbachev2001extremum}
rediscovered this theorem with a different proof.

In this section and the next, we will first discuss how the Delsarte
problem gives a packing bound function, and then we will generalize
the results to the Lasserre hierarchy from \cite{de2015semidefinite}.
We will denote the packing bound functions corresponding to
$\vartheta$, $\vartheta'$, and $\vartheta^+$ under
Theorems~\ref{theorem:pbffromsandwich} and~\ref{thetasandwich} by
$\vartheta$, $\vartheta'$, and $\vartheta^+$ again; this is an abuse
of notation, but it will not cause any actual ambiguity. We will also
define topological variants $\vartheta^{\topo}$,
$\vartheta'^{,\topo}$, and $\vartheta^{+,\topo}$ below, which will
impose continuity.

First, we need a few definitions. By a \emph{finite signed measure},
we mean a signed Borel measure $\mu$ on $\R^d$ such that $-\infty <
\mu(A) < \infty$ for every Borel set $A$ (bounded or not). If $\mu$
and $\mu_1,\mu_2,\dots$ are finite signed measures, we say $\mu_n$ is
weak-$*$ convergent to $\mu$ if every bounded, continuous function $f
\colon \R^d \to \R$ satisfies
\[
\int f \, d\mu_n \to \int f\, d\mu
\]
as $n \to \infty$.

A \emph{positive semidefinite kernel} on a bounded Borel set $C$ is a
function $K \colon C \times C \to \R$ such that $K(x,y)=K(y,x)$ and
for every finite subset $S$ of $C$ and function $w \colon S \to \R$,
\[
\sum_{x,y \in S} K(x,y) w(x) w(y) \ge 0.
\]
For example, $K(x,y) := f(x)f(y)$ is positive semidefinite for any
function $f \colon C \to \R$, and we denote this kernel by $f \otimes
f$.  If $K$ is continuous, then the inequality defining positive
semidefiniteness is equivalent to
\[\iint_{C \times C} K(x,y) \,d\mu(x) \,d\mu(y) \ge 0\]
for all finite signed measures $\mu$; here finite sets $S$ correspond
to $\mu$ with finite support (i.e., linear combinations of delta
functions supported at points), and finitely supported signed measures
are weak-$*$ dense among all finite signed measures.  A finite signed
measure $\nu$ is \emph{positive semidefinite} on $C \times C$ if
\[\iint_{C \times C} K(x,y) \,d\nu(x,y) \ge 0\]
for all continuous, positive semidefinite kernels $K$ on $C$.  For
simplicity, we often will call $\nu$ positive semidefinite on $C$,
rather than $C \times C$.  For finitely supported positive
semidefinite measures $\nu$, $\iint_{C \times C} K(x,y) \,d\nu(x,y)
\ge 0$ for all positive semidefinite kernels $K$, even if $K$ is not
continuous.

In this language, we can interpret the construction of packing bound
functions from Theorem~\ref{theorem:pbffromsandwich},
Definition~\ref{def:theta}, and Theorem~\ref{thetasandwich} as
follows. Let $\Delta(C)$ denote the diagonal $\{(x,x) : x \in C\}$ in
$C \times C$.

\begin{proposition}
For each bounded Borel subset $C$ of $\R^d$,
\begin{enumerate}
    \item $\vartheta^+(C)$ is the supremum of $\nu(C \times C)$ over
      finitely supported positive semidefinite signed measures $\nu$
      on $C$ such that $\nu(\Delta(C)) = 1$ and $\nu \le 0$ on the set
      of pairs $(x,y)$ such that $0< |x-y| < 2$,
    \item $\vartheta(C)$ is the supremum of $\nu(C \times C)$ over
      finitely supported positive semidefinite signed measures $\nu$
      on $C$ such that $\nu(\Delta(C)) = 1$ and $\nu = 0$ for pairs
      $(x,y)$ such that $0 < |x-y| < 2$, and
    \item $\vartheta'(C)$ is the supremum of $\nu(C \times C)$ over
      finitely supported positive semidefinite signed measures $\nu$
      on $C$ such that $\nu(\Delta(C)) = 1$, $\nu = 0$ for pairs
      $(x,y)$ such that $0 < |x-y| < 2$, and $\nu \ge 0$ everywhere.
\end{enumerate}
\end{proposition}

\begin{proof}
This follows immediately by considering matrices as finitely supported
signed measures on $C \times C$.  Every measure $\nu$ as above gives a
feasible solution for the corresponding sandwich function
$\vartheta^+$, $\vartheta$, or $\vartheta'$ over some induced finite
subgraph.  Conversely, any feasible solution over an induced finite
subgraph gives a feasible $\nu$.  Taking the supremum over feasible
solutions on both sides yields the result.
\end{proof}

Similarly, the dual semidefinite programs are as follows:
\begin{enumerate}
\item $\vartheta'(C)^*$ is the infimum of $t$ over all positive
  semidefinite kernels $K$ on finite subsets $S$ of $C$ such that
  $K(x,x) = t-1$ for all $x \in S$ and $K(x,y) \le -1$ when $|x-y| \ge
  2$,
\item $\vartheta(C)^*$ is the infimum of $t$ over all positive
  semidefinite kernels $K$ on finite subsets $S$ of $C$ such that
  $K(x,x) = t-1$ for all $x \in S$ and $K(x,y) = -1$ when $|x-y| \ge
  2$, and
\item $\vartheta^+(C)^*$ is the infimum of $t$ over all positive
  semidefinite kernels $K$ on finite subsets $S$ of $C$ such that
  $K(x,x) = t-1$ for all $x \in S$, $K(x,y) \ge -1$ for all $x,y \in
  S$, and $K(x,y) = -1$ when $|x-y| \ge 2$.
\end{enumerate}
It follows from Proposition~\ref{discretetheta} that $\vartheta'(C) =
\vartheta'(C)^*$, $\vartheta(C) = \vartheta(C)^*$, and $\vartheta^+(C)
= \vartheta^+(C)^*$, and that we can take $S=C$ in the dual programs
when $C$ is finite.

The description in terms of finitely supported signed measures
suggests that we should consider the analogous problems using
arbitrary finite signed measures instead.  Such measures have the
advantage that one can average feasible solutions over the action of a
compact Lie group of isometries and arrive at invariant semidefinite
programs for packing problems \cite{bachoc2012invariant}.  We will
refer to these as the topological analogues of the discrete
$\vartheta^+$, $\vartheta$, and $\vartheta'$, and denote them by
$\vartheta^{+,\topo}$, $\vartheta^{\topo}$, and $\vartheta'^{,\topo}$,
respectively. In other words,
\begin{enumerate}
    \item $\vartheta^{+,\topo}(C)$ is the supremum of $\nu(C \times
      C)$ over finite, positive semidefinite signed measures $\nu$ on
      $C$ such that $\nu(\Delta(C)) = 1$ and $\nu \le 0$ on the set of
      pairs $(x,y)$ such that $0< |x-y| < 2$,
    \item $\vartheta^{\topo}(C)$ is the supremum of $\nu(C \times C)$
      over finite, positive semidefinite signed measures $\nu$ on $C$
      such that $\nu(\Delta(C)) = 1$ and $\nu = 0$ for pairs $(x,y)$
      such that $0 < |x-y| < 2$, and
    \item $\vartheta'^{,\topo}(C)$ is the supremum of $\nu(C \times
      C)$ over finite, positive semidefinite signed measures $\nu$ on
      $C$ such that $\nu(\Delta(C)) = 1$, $\nu = 0$ for pairs $(x,y)$
      such that $0 < |x-y| < 2$, and $\nu \ge 0$ everywhere.
\end{enumerate}

The topological analogue of the dual semidefinite programs for
$\vartheta^+$, $\vartheta$, and $\vartheta'$ uses continuous kernels.
We define
\begin{enumerate}
\item $\vartheta'^{,\topo}(C)^*$ to be the infimum of $t$ over all
  continuous, positive semidefinite kernels $K$ on $C$ such that
  $K(x,x)\le t-1$ for all $x$ and $K(x,y)\le -1$ when $|x-y| \ge 2$,
\item $\vartheta^{\topo}(C)^*$ to be the infimum of $t$ over all
  continuous, positive semidefinite kernels $K$ on $C$ such that such
  that $K(x,x)\le t-1$ for all $x$ and $K(x,y) = -1$ when $|x-y| \ge
  2$, and
\item $\vartheta^{+,\topo}(C)^*$ to be the infimum of $t$ over all
  continuous, positive semidefinite kernels $K$ on $C$ such that such
  that $K(x,x)\le t-1$ for all $x$, $K(x,y) = -1$ when $|x-y| \ge 2$,
  and $K(x,y) \ge -1$ everywhere.
\end{enumerate}
Note that here we require just $K(x,x) \le t-1$, rather than $K(x,x) =
t-1$ as in the finite case (Proposition~\ref{discretetheta}).  That
change makes no difference in the finite case, but it will be
convenient in Proposition~\ref{thetaduality}.

The following proposition is the topological analogue of
Proposition~\ref{discretetheta}.

\begin{proposition}[Weak duality]
For each bounded Borel set $C$,
\[
\vartheta'^{,\topo}(C) \le \vartheta'^{,\topo}(C)^*, \quad
\vartheta^{\topo}(C) \le \vartheta^{\topo}(C)^*, \quad\text{and}\quad
\vartheta^{+,\topo}(C) \le \vartheta^{+,\topo}(C)^*.
\]
\end{proposition}

\begin{proof}
Let $A$ be $\vartheta'^{,\topo}$, $\vartheta^{\topo}$, or
$\vartheta^{+,\topo}$. In each case, the hypotheses on $\nu$, $t$, and
$K$ in the definitions of $A(C)$ and $A(C)^*$ show that
\[
\nu(C \times C) = \iint_{C \times C} d\nu(x,y) \le \iint_{C \times C}
\big(1+K(x,y)\big) \, d\nu(x,y) \le t \,\nu(\Delta(C)) = t,
\]
as desired.
\end{proof}

The following proposition clarifies the relationship between the
topological and discrete invariants.

\begin{proposition}\label{thetaduality}
Let $C \subseteq \R^d$ be compact, and let $C(\varepsilon)$ denote an
$\varepsilon$-neighborhood around $C$.  Then
\[\vartheta'^{,\topo}(C) = \lim_{\varepsilon \to 0^+} \vartheta'(C(\varepsilon)).\]
In addition, for any $\varepsilon > 0$,
\[
\begin{split}
\vartheta'(C) \le \vartheta'^{,\topo}(C) &\le \vartheta'^{,\topo}(C)^*
\le \vartheta'(C(\varepsilon)), \\ \vartheta^{+}(C) \le
\vartheta^{+,\topo}(C) &\le \vartheta^{+,\topo}(C)^* \le
\vartheta^+(C(\varepsilon)), \text{ and}\\ \vartheta(C) \le
\vartheta^{\topo}(C) &\le \vartheta^{\topo}(C)^* \le
\vartheta(C(\varepsilon)).
\end{split}
\]
\end{proposition}

We do not know whether there is any compact set $C \subseteq \R^d$ for
which $\vartheta'(C) \ne \vartheta'^{,\topo}(C)$.

\begin{corollary}[Strong duality]
For each bounded Borel set $C$,
\[
\vartheta'^{,\topo}(C) = \vartheta'^{,\topo}(C)^*.
\]
\end{corollary}

This corollary follows directly from Proposition~\ref{thetaduality},
and will be generalized to the Lasserre hierarchy in
Theorem~\ref{lasstrongduality}. We do not know whether strong duality
holds for $\theta$ or $\theta^+$, but they are less significant as
packing bounds.

The proof of Proposition~\ref{thetaduality} will use the following
lemma, which allows us to reduce a problem on an infinite topological
space to a problem on a finite simplicial complex.

\begin{lemma}\label{geometricsimplicial}
For every compact subset $C \subseteq \R^d$ and $\varepsilon>0$, there
is a finite, pure simplicial $d$-complex $Y$ embedded in $\R^d$ such
that $C \subseteq Y \subseteq C(\varepsilon)$ and every simplex in $Y$
has diameter less than $\varepsilon$.
\end{lemma}

Here, saying $Y$ is finite means the complex is built from finitely
many simplices in $\R^d$, not that its geometric realization is a
finite set.

\begin{proof}
Given a triangulation of $\R^d$ with simplices of diameter less than
$\varepsilon$, let $Y$ consist of the simplices that are contained in
$C(\varepsilon)$. To obtain such a triangulation, we can start with a
decomposition of a cube of diameter less than $\varepsilon$ into $d!$
simplices whose vertices are vertices of the cube, and then extend the
decomposition to $\R^d$ via reflection across the faces of the cube.
\end{proof}

\begin{proof}[Proof of Proposition~\ref{thetaduality}]
We first prove that
\[\limsup_{\varepsilon\to0^+} \vartheta'(C(\varepsilon)) \le
\vartheta'^{,\topo}(C).\]
Choose a sequence $\varepsilon_n \to 0$ and $\varepsilon_n>0$ and
feasible solutions $\nu_n$ for $\vartheta'(C(\varepsilon_n))$.  Since
all these measures are nonnegative, their supports are bounded, and
their total measures are bounded by $\vartheta'(C(\max_n
\varepsilon_n))$, there is a weak-$*$ convergent subsequence. By
passing to such a subsequence, we can assume that the entire sequence
is weak-$*$ convergent.

Let $\nu$ be the weak-$*$ limit of $\nu_n$; then $\nu$ is a
nonnegative measure because each $\nu_n$ is nonnegative.  It must be
supported on $\bigcap_n C(\varepsilon_n)$, which is equal to $C$ by
the assumption that $C \subseteq \R^d$ is compact.  It must be
positive semidefinite because every continuous function $f$ on $C$ can
be extended to a continuous function $f'$ on $\R^d$ with compact
support, and $\nu(f \otimes f)=\nu(f' \otimes f')$, which is the limit
of the sequence $\nu_n(f' \otimes f')$, each term of which is
nonnegative.  Finally, $\nu$ must be supported on $\{(x,y) : x=y
\text{ or } |x-y| \ge 2\}$, because $\{(x,y) : 0 < |x-y| < 2\}$ is an
open set and thus
\[
\begin{split}
\nu(\{(x,y) : 0 < |x-y| < 2\}) &\le \liminf_{n \to \infty}
\nu_n(\{(x,y) \in C \times C : 0 < |x-y| < 2\})\\ &= \liminf_{n \to
  \infty} 0 = 0.
\end{split}
\]
Therefore $\nu$ is a feasible solution for $\vartheta'^{,\topo}(C)$.
We also note that the objective $\nu(C \times C) = \nu(\R^d \times
\R^d)$ is a continuous functional for the weak-$*$ topology (because
it is the integral of the constant function $1$). By taking $\nu_n$ to
satisfy $\nu_n(\R^d \times \R^d) \ge \vartheta'(C(\varepsilon_n)) -
\delta$ for very small $\delta > 0$, we can guarantee that
\[\limsup_{\varepsilon\to0^+} \vartheta'(C(\varepsilon)) -
\delta \le \vartheta'^{,\topo}(C).\]
The desired inequality follows.

Let $A = \vartheta$, $\vartheta^+$, or $\vartheta'$.  It remains to
prove that
\[A(C) \le A^{\topo}(C) \le A^{\topo}(C)^* \le A(C(\varepsilon)) \]
for each $\varepsilon > 0$.  The first inequality follows immediately
from the definitions, and the second by weak duality in
Proposition~\ref{thetaduality}, so we just need to show the last
inequality.  In fact, we will prove an upper bound of
$A((1+\varepsilon) C(\varepsilon))$. That bound looks a little weaker,
but $(1+\varepsilon) C(\varepsilon)$ is contained in a neighborhood of
$C(\varepsilon)$ of radius $\varepsilon \sup_{x \in C(\varepsilon)}
|x| = \varepsilon^2 + \varepsilon \sup_{x \in C} |x|$, and therefore
$A((1+\varepsilon) C(\varepsilon)) \le A(C(\delta))$, where $\delta =
\varepsilon(1+\varepsilon+\sup_{x \in C} |x|)$. Thus, we can obtain
the desired upper bound simply by decreasing $\varepsilon$, to make
$\delta$ as small as we want. To prove the upper bound of
$A((1+\varepsilon) C(\varepsilon))$, we will use the dual semidefinite
program given in Proposition~\ref{discretetheta}, and we will relate
it to $A^{\topo}(C)^*$ by discretizing space using a suitable
simplicial complex.

Let $Y$ be a finite, pure simplicial $d$-complex such that $C
\subseteq Y \subseteq C(\varepsilon)$ and all simplices in $Y$ have
diameter less than $\varepsilon$, as in
Lemma~\ref{geometricsimplicial}, and let $Y_0$ be the set of vertices
of simplices in $Y$.  Each point $x$ in $Y$ can be written using
barycentric coordinates as $x = \sum_i a_i x_i$ with $\{ x_i
\}_{i=0}^d \subseteq Y_0$ defining a top-dimensional simplex in $Y$,
$\sum_i a_i = 1$, and $a_i \ge 0$.  Moreover, such an expression is
unique in the sense that two such expressions for $x$ yield simplices
that intersect in a face containing $x$, and $a_i = 0$ for vertices
$x_i$ not contained in the closure of that face.

Now given a positive semidefinite kernel $K$ on $Y_0$, define a kernel
$L$ on all of $Y$ by setting
\[ L(x,y) = L\mathopen{}\left( \sum_{i=0}^d a_i x_i, \sum_{j=0}^d b_j y_j
\right)\mathclose{} = \sum_{i,j=0}^d a_i b_j K(x_i, y_j)
\]
given barycentric coordinates $x=\sum_i a_ix_i$ and $y=\sum_j b_jy_j$
as above. The uniqueness of these coordinates show that $L$ is well
defined and continuous. To show that $L$ is positive semidefinite, we
must show that
\[
\sum_{m,n} w_m w_n L(x_m,x_n) \ge 0
\]
for each finite set of points $x_n$ in $Y$ with weights $w_n$. If we
use barycentric coordinates $x_n = \sum_i a_{i,n} x_{i,n}$, then
\[
\sum_{m,n} w_m w_n L(x_m,x_n) = \sum_{m,n,i,j} w_m w_n a_{i,m} a_{j,n}
K(x_{i,m},x_{j,n}),
\]
which is indeed nonnegative because $K$ is positive
semidefinite. Thus, $L$ is a continuous, positive semidefinite kernel
on $Y$. The remaining argument will split into cases, depending on
whether we are analyzing $\vartheta'$, $\vartheta$, or $\vartheta^+$.

First, suppose $K$ is a feasible solution for $\vartheta'(Y_0)^*$,
with $K(x,x)=t-1$ for all $x$.  Then $L(x,y)$ is a convex combination
of values $K(x_i, y_j)$ with $\max(|x-x_i|,|y-y_j|)\le \varepsilon$
and therefore $L(x,y) \le -1$ for $|x-y| \ge 2+2\varepsilon$.
Furthermore, $K$ takes its maximum value on the diagonal, because it
is a positive semidefinite kernel, and thus $L(x,x) \le t-1$. This
construction shows that
\[
\vartheta'^{,\topo}\mathopen{}\left(\frac{Y}{1+\varepsilon}\right)^*\mathclose{}
\le \vartheta'(Y_0)^* = \vartheta'(Y_0) \le
\vartheta'(C(\varepsilon)),
\]
where the rescaling by a factor of $1+\varepsilon$ takes care of the
$\varepsilon$ in the inequality $|x-y| \ge 2+2\varepsilon$ above, and
rescaling space implies that
\[
\vartheta'^{,\topo}(Y)^* \le \vartheta'((1+\varepsilon)
C(\varepsilon)).
\]
By restricting $L$ to $C$, we conclude that $\vartheta'^{,\topo}(C)^*
\le \vartheta'((1+\varepsilon) C(\varepsilon))$, which concludes the
proof.

Suppose instead that $K$ is a feasible solution for
$\vartheta(Y_0)^*$.  Then by the same reasoning, $L(x,y) = -1$ for
$|x-y| \ge 2 + 2\varepsilon$, so $\vartheta^{\topo}(C)^* \le
\vartheta((1+\varepsilon) C(\varepsilon))$.

Finally, suppose $K$ is a feasible solution for $\vartheta^+(Y_0)^*$.
Then $L(x,y) \ge -1$ everywhere since it is a convex combination of
$K(x_i, y_j) \ge -1$, so $\vartheta^{+,\topo}(C)^* \le
\vartheta((1+\varepsilon) C(\varepsilon))$.
\end{proof}

Our strategy for computing the Euclidean limits will be to compute the
corresponding Euclidean limits of the topological analogues.  Even
though we do not know whether they are packing bound functions, their
limits are still well defined.  The following lemma shows that the
Euclidean limits of the discrete and topological versions of a
sandwich function are the same.

\begin{lemma}\label{cubetoplimit}
Let $A$ be a packing bound function, and let $A^{\topo}$ be a
real-valued function defined on compact subsets of Euclidean space and
satisfying
\[A(C) \le A^{\topo}(C) \le A^{\topo}(C)^* \le A(C(\varepsilon))\]
whenever $\varepsilon>0$.  Letting $I^n = [0,1]^n$ be the unit cube in
$\R^n$ as usual,
\[\lim_{r \to \infty} \frac{A^{\topo}(rI^n)}{r^n} = \lim_{r \to \infty}
\frac{A^{\topo}(rI^n)^*}{r^n} = \delta_{A, n}.\]
\end{lemma}

\begin{proof}
This follows immediately by noting that
\[A(rI^n) \le A^{\topo}(rI^n) \le A^{\topo}(rI^n)^* \le A((r+1)I^n). \qedhere\]
\end{proof}

We are ready for the main result of this section, an explicit
description of the Euclidean limits of $\vartheta'$, $\vartheta$, and
$\vartheta^+$.

\begin{theorem}\label{thetalimit}
The quantity $\mathcal{L}_n(B_1^n) \delta_{\vartheta',n}$ is the
linear programming bound for the sphere packing density in $\R^n$,
while $\mathcal{L}_n(B_1^n)\delta_{\vartheta, n} =
\mathcal{L}_n(B_1^n)\delta_{\vartheta^+, n} = 1$.
\end{theorem}

The key step in the proof of the theorem is the following
lemma. Recall that a function $f \colon \R^n \to \R$ is \emph{positive
semidefinite} if the kernel $K$ defined by $K(x,y)=f(x-y)$ is positive
semidefinite. In particular, $f$ must be an even function.

\begin{lemma}
The Euclidean limits of $\vartheta'$, $\vartheta$, and $\vartheta^+$
are characterized as follows:
\begin{enumerate}
\item $\delta_{\vartheta', n}$ is the infimum of $f(0) /
  \widehat{f}(0)$ over all continuous, integrable, positive
  semidefinite functions $f\colon \R^n \to \R$ such that $f(x) \le 0$
  for $|x| \ge 2$ and $\widehat{f}(0)>0$,
\item $\delta_{\vartheta, n}$ is the infimum of $f(0) /
  \widehat{f}(0)$ over all feasible solutions to $\delta_{\vartheta',
    n}$ that additionally satisfy $f(x) = 0$ for $|x| \ge 2$, and
\item $\delta_{\vartheta^+,n}$ is the infimum of $f(0) /
  \widehat{f}(0)$ over all feasible solutions to $\delta_{\vartheta,
    n}$ that additionally satisfy $f(x) \ge 0$ for $|x| \le 2$.
\end{enumerate}
\end{lemma}

Note that without loss of generality, we can assume these functions
$f$ are radial functions (by averaging over all rotations), because
the constraints and objective functions are radially
symmetric. Furthermore, we can assume that $f$ has compact support;
this is automatic for $\vartheta$ and $\vartheta^+$, and can be proved
as follows for $\vartheta'$ by mollifying a feasible function $f$. The
convolution $1_{B_{R/2}^n(0)} * 1_{B_{R/2}^n(0)}$ is a continuous
function supported in $B_R^n(0)$, and it is positive semidefinite
since its Fourier transform is
$\big(\widehat{1}_{B_{R/2}^n(0)}\big)^2$. If we normalize it by
setting $g_R := \big(1_{B_{R/2}^n(0)} *
1_{B_{R/2}^n(0)}\big)/\mathcal{L}_n(B_{R/2}^n(0))$ so that $g_R(0)=1$,
then $g_R$ converges pointwise to $1$ everywhere as $R \to \infty$. In
particular, the product $f_R := f \cdot g_R$ is continuous and
supported in $B_R^n(0)$, it is positive semidefinite by the Schur
product theorem (Theorem~7.5.3 in \cite{HJ}), and $\lim_{R \to \infty}
\widehat{f}_R(0) = \widehat{f}(0)$ by dominated convergence. By
replacing $f$ with $f_R$, we can come arbitrarily close to the ratio
$f(0)/\widehat{f}(0)$ using compactly supported functions.

\begin{proof}
We will apply Lemma~\ref{cubetoplimit} and
Proposition~\ref{thetaduality} throughout.  Let $K$ be a continuous,
positive semidefinite kernel on $rI^n \subseteq \R^n$, such that $K$
is a feasible solution for one of the dual topological bounds. We
extend it by $0$ to give a positive semidefinite kernel on $\R^n$.
Define $\widetilde{f} \colon \R^n \times \R^n \to \R$ by the formula
\[ \widetilde{f}(x,y) = \frac{1}{r^{2n}} \int_{\R^n} K(x+z,y+z) +
1_{rI^n \times rI^n}(x+z,y+z) \,dz. \]
Then $\widetilde{f}$ is translation invariant (i.e.,
$\widetilde{f}(x+t,y+t)=\widetilde{f}(x,y)$) and compactly supported,
and it is a continuous function because translating integrable
functions is continuous under the $L^1$ norm. Thus, the function $f
\colon \R^n \to \R$ given by $\widetilde{f}(x,y)=f(x-y)$ is well
defined and a continuous, compactly supported function. The function
$\widetilde{f}$ is a positive semidefinite kernel because it is an
integral of such kernels, and so $f$ is positive semidefinite by
definition. We have
\[
\widehat{f}(0) = r^{-2n} \iint_{\R^n \times \R^n} K(x,y) \, dx \, dy +
1 \ge 1,
\]
and if $K(x,x) = t-1$ for all $x$, then
\[
f(0) = \frac{t}{r^n}.
\]

Now suppose $K$ is feasible for $\vartheta'^{,\topo}(rI^n)^*$.  Then
$\widetilde{f}(x,y) \le 0$ for $|x-y| \ge 2$, because the integrand
defining $\widetilde{f}(x,y)$ cannot be positive in this case.  Thus,
the infimum over $f$ satisfying the conditions in the lemma statement
is bounded above by $r^{-n} \vartheta'^{,\topo}(rI^n)^*$.

Suppose $K$ is feasible for $\vartheta^{\topo}(rI^n)^*$.  Then
$\widetilde{f}(x,y) = 0$ for $|x-y| \ge 2$, because the integrand
defining $\widetilde{f}(x,y)$ is identically $0$.

Finally, suppose $K$ is feasible for $\vartheta^{+,\topo}(rI^n)^*$.
Then $\widetilde{f}(x,y) \ge 0$ for $|x-y| \le 2$, because the
integrand defining $\widetilde{f}(x,y)$ is nonnegative.  Therefore, we
conclude one inequality for each of the statements in the lemma.

For the other direction, suppose $f \colon \R^n \to \R$ is continuous,
positive semidefinite, and integrable, with $\widehat{f}(0)>0$. As
noted after the lemma statement, we can also assume that $\supp(f)
\subseteq B_R^n(0)$ for some $R$.

It is not hard to show that
\[ \lim_{r\to\infty} \frac{\iint_{rI^n \times rI^n} f(x-y) \,dx \,dy}{r^n}
= \widehat{f}(0). \]
Fix $\varepsilon>0$, and let $r$ be large enough that
\begin{equation} \label{eqLrInlimit}
\iint_{rI^n \times rI^n} f(x-y) \,dx \,dy \le r^n (1+\varepsilon)
\widehat{f}(0).
\end{equation}
and $r > 2R$.  Let $J^n = [R,r-R]^{n} \subseteq rI^n$, and define $K
\colon J^n \times J^n \to \R$ by
\[ K(x,y) = f(x-y) - \frac{1-\varepsilon}{r^n} \widehat{f}(0). \]
We claim that $K$ is positive semidefinite.  Let $S$ be a finite
subset of $J^n$, and let $w_x$ be real weights for $x \in S$.  It
suffices to show that
\begin{equation}
\label{eq:Kpsd}
\sum_{x,y \in S} w_x w_y K(x,y) \ge 0.
\end{equation}
To do so, we define a signed measure $\nu$ by
\[ \sum_{x \in S} w_x \delta_x - \frac{\sum_{x \in S} w_x}{r^n} \mu_r, \]
where $\delta_x$ is a unit mass placed at the point $x$ and $\mu_r$ is
the Lebesgue measure on $rI^n$.  Positive semidefiniteness of $f$
implies
\[ \iint f(x-y) \,d\nu(x) \,d\nu(y) \ge 0, \]
and the left side of this inequality is equal to
\begin{equation} \label{eq:psdf}
  \begin{split}
    &\sum_{x,y \in S} w_x w_y f(x-y) +
    \frac{\left( \sum_{x\in S} w_x \right)^2}{r^{2n}} \iint_{rI^n \times rI^n} f(x-y) \,dx \,dy\\
    & \quad- \frac{2}{r^n} \left(\sum_{x \in S} w_x \right) \sum_{x \in S} w_x\int_{rI^n} f(x-y) \, dy.
  \end{split}
\end{equation}
Because $S \subseteq [R,r-R]^n$ and $\supp(f) \subseteq B_R^n(0)$,
\[
\int_{rI^n} f(x-y) \, dy = \widehat{f}(0)
\]
for each $x \in S$. If we substitute this identity and
\eqref{eqLrInlimit} into \eqref{eq:psdf}, we conclude that
\[ \sum_{x,y \in S} w_x w_y f(x-y) - \left( \sum_{x \in S} w_x \right)^2
\frac{1-\varepsilon}{r^n} \widehat{f}(0)  \ge 0, \]
which is equivalent to \eqref{eq:Kpsd}. Thus, $K$ is positive
semidefinite.

Because $\widehat{f}(0) > 0$, we can rescale $f$ so that
\[ \widehat{f}(0) = \frac{r^n}{1-\varepsilon}. \]
Then $K(x,y) = f(x-y) - 1$ is positive semidefinite on $J^n$, and in
particular $K(x,x) = f(0) - 1$.

Now suppose that $f$ satisfies $f(x) \le 0$ for $|x| \ge 2$, as in the
case of $\vartheta'$.  Then $K(x,y) \le -1$ for $|x-y| \ge 2$, so $K$
is feasible for $\vartheta'^{,\topo}(J^n)^*$.  The bound we get on
$\vartheta'^{,\topo}(J^n)^*$ is
\[ f(0) = \frac{r^n}{1-\varepsilon} \frac{f(0)}{\widehat{f}(0)}. \]
Dividing by $(r-2R)^n$ and letting $r \to \infty$ and then
$\varepsilon \to 0$ allows us to conclude that
\[ \lim_{r\to\infty} \frac{\vartheta'^{,\topo}((r-2R)I^n)^*}{(r-2R)^n}
\le \frac{f(0)}{\widehat{f}(0)} \]
and the left side is equal to $\delta_{\vartheta', n}$. This completes
the proof for the case of $\vartheta'$.

Suppose instead that $f(x) = 0$ for $|x| \ge 2$, as in the case of
$\vartheta$.  Then $K(x,y) = -1$ for $|x-y| \ge 2$, and so $K$ is
feasible for $\vartheta^{\topo}(J^n)^*$.  By the same reasoning as
above, we conclude that
\[ \delta_{\vartheta,n} \le \frac{f(0)}{\widehat{f}(0)}. \]

Finally, suppose that $f(x) \ge 0$ for $|x| \le 2$, as in the case of
$\vartheta^+$.  Then $K(x,y) \ge -1$ for $|x-y| \le 2$, and so $K$ is
feasible for $\vartheta^{+,\topo}(J^n)^*$.  We conclude that
\[ \delta_{\vartheta^+,n} \le \frac{f(0)}{\widehat{f}(0)}, \]
which completes the proof.
\end{proof}

\begin{proof}[Proof of Theorem~\ref{thetalimit}]
The lemma directly gives the statement for $\delta_{\vartheta', n}$
and the linear programming bound. All that remains is to compute the
optimal solutions to the following problems:
\begin{enumerate}
\item Minimize $f(0) / \widehat{f}(0)$ over all continuous, positive
  semidefinite functions $f$ such that $f(x) = 0$ for $|x| \ge 2$ and
  $\widehat{f}(0)>0$.
\item Minimize $f(0) / \widehat{f}(0)$ over all continuous, positive
  semidefinite functions $f$ such that $f(x) \ge 0$ for $|x| \le 2$,
  $f(x) = 0$ for $|x| \ge 2$, and $\widehat{f}(0)>0$.
\end{enumerate}
These problems were solved by Gorbachev \cite{gorbachev2001extremum},
and Siegel had given the same answer to essentially the same problems
in \cite{siegel}. For the convenience of the reader, we will sketch
Gorbachev's proof.

The optimal function $f$ is the convolution
\[ f = 1_{B_1^n(0)} * 1_{B_1^n(0)} \]
and the objective is $f(0) / \widehat{f}(0) =
\mathcal{L}_n(B_1^n)^{-1}$. The proof of optimality is a consequence
of the quadrature formula of Ben Ghanem and Frappier
\cite{ghanem1998explicit}, which incidentally also shows that the
Levenshtein bound is optimal among certain band-limited solutions to
the linear programming bound problem \cite{gorbachev2000extremum,
  cohn2002new}.  Specifically, we use the $p=0$ case of Lemma~4 in
\cite{ghanem1998explicit}, under the minimal hypotheses established by
Grozev and Rahman \cite{grozev1995quadrature} (which are not stated
explicitly for this identity in \cite{ghanem1998explicit} but follow
from the same proof). For any continuous, radial function $f \colon
\R^n \to \R$ supported on $B_r^n(0)$ whose Fourier transform is
integrable, the formula says that
\[ f(0) = \frac{1}{\mathcal{L}_n(B_{r/2}^n)} \widehat{f}(0) +
\sum_{m=1}^\infty \alpha_m \widehat{f} \left( \frac{\lambda_m}{\pi r} \right), \]
where the node points $\lambda_m$ are the positive roots of the Bessel
function $J_{n/2}$ and the coefficients $\alpha_m$ are certain
explicit quantities with $\alpha_m>0$.  Any feasible solution $f$ to
problems~(1) or~(2) above satisfies these hypotheses with $r=2$ and
has $\widehat{f} \ge 0$ (see, for example, Corollary~1.26 in Chapter~I
of \cite{SteinWeiss} for why $\widehat{f}$ is integrable).  Since the
infinite sum is nonnegative, we conclude that
\[ \frac{f(0)}{\widehat{f}(0)} \ge \mathcal{L}_n(B_1^n)^{-1}. \qedhere \]
\end{proof}

\begin{remark}
In the one-dimensional case,
\[
1 = \mathcal{L}_1(B_1^1) \delta_{\pack,1} \le \mathcal{L}_1(B_1^1)
\delta_{\vartheta',1} \le \mathcal{L}_1(B_1^1) \delta_{\cov,1} = 1.
\]
In other words, the fact that sphere covering and sphere packing have
the same density in $\R^1$ implies that the linear programming bound
must be sharp in that case. Of course this fact can be proved directly
by exhibiting a closed-form auxiliary function, but it is interesting
to prove it without the need to construct any explicit auxiliary
function.
\end{remark}

\section{The Lasserre hierarchy for sphere packing}

The Lasserre hierarchy \cite{Lasserre} is an important family of
semidefinite relaxations of the independence number; this hierarchy
starts with $\vartheta'$ and extends it to successively stronger
bounds, which converge to the exact independence number. Based on
foundations laid by Laurent \cite{Laurent} and by de Laat and
Vallentin \cite{de2015semidefinite}, in this section we show that the
Lasserre hierarchy consists of sandwich functions, and we write their
Euclidean limits as optimization problems. The net result is a
generalization of the linear programming bound to a hierarchy of
bounds that converge to the exact sphere packing density.

\subsection{Review of bounds for compact spaces}

A \emph{topological packing graph} is a graph and a Hausdorff
topological space such that every finite clique is contained in an
open clique. Equivalently, each vertex and each edge is contained in
an open clique.  Every graph is a topological packing graph with the
discrete topology, and every compact topological packing graph has
finite independence number (and furthermore finite clique covering
number, which is an even stronger assertion). The edges in a
topological packing graph indicate pairs of vertices that are too
close together to appear in the same packing, and independent sets
correspond to packings.

The Lasserre hierarchy on compact topological packing graphs was
introduced by de Laat and Vallentin \cite{de2015semidefinite}.  We
begin by reviewing this hierarchy.

\begin{definition}
We set the following notation:
\begin{enumerate}
\item For a topological space $V$, let $C(V)$ be the set of continuous
  functions from $V$ to $\R$, let $C_0(V) \subseteq C(V)$ consist of
  the functions that vanish at infinity, and let $C_c(V) \subseteq
  C(V)$ consist of those with compact support.
    \item For a locally compact Hausdorff space $V$, let $\M_+(V)$
      denote the set of Radon measures on $V$ (i.e., regular Borel
      measures), and let $\M_\pm(V)$ denote the set of finite, regular
      signed Borel measures on $V$.  By the Riesz representation
      theorem, $\M_\pm(V) = C_0(V)^*$ under the pairing given by
      integration (see, for example
      \cite[Theorem~7.3.6]{cohnmeasuretheory}).
    \item Given a compact topological packing graph $G$, let
      $I_{=t,G}$ be the set of independent sets in $V(G)$ of size
      exactly equal to $t$, with the topology given as a subset of the
      quotient of $V(G)^t$ under the map $(v_1,\dots,v_t) \mapsto
      \{v_1,\dots,v_t\}$.  In particular, if $t > \alpha(G)$ then this
      set is empty, and $I_{=0,G} = \{\emptyset\}$.
    \item Given a compact topological packing graph $G$, let $I_{t,G}$
      be the set of independent sets of size at most $t$ in $G$, with
      the topology given by the disjoint union of $I_{=k,G}$ for $k =
      0, \dots, t$.
\end{enumerate}
\end{definition}

We often write $I_t$ instead of $I_{t,G}$ when the context is clear.
We can think of $I_t$ as a moduli space of packings on $G$ by
independent sets of size at most $t$, and we will often use the
covariant functoriality of these moduli spaces for morphisms in the
category $\overline{\Graph}$.

When $G$ is a compact topological packing graph, the space $I_t$ is
compact by Lemma~1 in \cite{de2015semidefinite}, and we will make
frequent use of the duality between $C(I_t)$ and $\M_\pm(I_t)$. A key
role will be played by an operator
\[
A_t \colon C(I_t \times I_t) \to C(I_{2t}).
\]
For $f \in C(I_t \times I_t)$ and $S \in I_{2t}$, the function $A_tf$
is defined by
\[A_t f(S) = \sum_{\substack{J \cup J' = S \\ J,J' \in I_t}} f(J, J').\]
We define the operator $A_t^*\colon \M_\pm(I_{2t}) \to \M_{\pm}(I_t
\times I_t)$ to be the adjoint of $A_t$.  Specifically, for $\mu \in
\M_\pm(I_{2t})$, the signed measure $A_t^*\mu$ is characterized by
\begin{equation}
\label{eq:charac}
\iint_{I_{t} \times I_t} f(J, J') \,dA_t^*\mu(J, J') = \int_{I_{2t}}
\sum_{\substack{J \cup J' = S \\ J,J' \in I_t}} f(J, J') \, d\mu(S)
\end{equation} 
for all $f \in C(I_t \times I_t)$.

There are corresponding notions of positive semidefiniteness for
kernels and measures on $I_t$.  A kernel $K \colon I_t \times I_t \to
\R$ is positive semidefinite if for every finite subset $S \subseteq
I_t$, the matrix $(K(J, J'))_{J,J' \in S}$ is positive semidefinite.
Equivalently, $K(J,J')=K(J',J)$ and for every finite subset $S
\subseteq I_t$ and any weights $w_J \in \R$ for $J \in S$,
\[ \sum_{J,J' \in S} w_J w_{J'} K(J,J') \ge 0. \]
A signed measure $\mu$ on $I_t \times I_t$ is positive semidefinite if
every continuous, positive semidefinite kernel $K \colon I_t \times
I_t \to \R$ satisfies
\[
\iint_{I_t \times I_t} K(J,J') \, d\mu(J,J') \ge 0.
\]
Note that for compact topological packing graphs, $I_t$ is compact. In
more general cases we use only kernels with compact support.  By
Mercer's theorem \cite[Theorem~3.11.9]{SimonVol4}, an equivalent
condition is that for every continuous function $f\colon I_t \to \R$,
\[\mu(f \otimes f) := \iint_{I_t \times I_t} f(J) f(J') \,d\mu(J, J') \ge 0.\]
When $G$ is a finite graph (which must have the discrete topology,
because it must be Hausdorff), a symmetric measure $\mu \in \M_\pm(I_t
\times I_t)$ is positive semidefinite if and only if the function
$(J,J') \mapsto \mu(\{J\} \times \{J'\})$ is a positive semidefinite
kernel, because the cone of positive semidefinite matrices is
self-dual.

\begin{definition}
For a compact topological packing graph $G$, we define
$\las^{\topo}_t(G)$ to be the supremum of $\lambda(I_{=1})$ over all
$\lambda \in \M_\pm(I_{2t})$ such that $A_t^* \lambda$ is positive
semidefinite as a measure on $I_t \times I_t$ and $\lambda(\{
\emptyset \}) = 1$. We define $\las'^{,\topo}_t(G)$ the same way, with
the additional requirement that $\lambda$ be a positive measure.
\end{definition}

\begin{remark}
In \cite{de2015semidefinite}, de Laat and Vallentin use the notation
$\las_t(G)$ to refer to what we call the hierarchy
$\las'^{,\topo}_t(G)$.  However, the notation we use here is more
consistent with the convention of adding the $'$ for optimization over
positive measures, as in the relationship between $\vartheta$ and
$\vartheta'$. The ``$\topo$'' indicates the dependence on the topology
of $G$, along the lines of $\vartheta'$ and $\vartheta'^{,\topo}$.
\end{remark}

\begin{proposition}
For every compact topological packing graph $G$, $\las^{\topo}_t(G)$
and $\las'^{,\topo}_t(G)$ are nonincreasing in $t$ and satisfy
\[ \las^{\topo}_{2t}(G) \le \las'^{,\topo}_{t}(G) \le \las^{\topo}_t(G). \]
\end{proposition}

\begin{proof}
To see that these quantities are nonincreasing, consider restricting a
measure $\lambda$ on $I_{2(t+1)}$ to $I_{2t}$.  If $\lambda$ is
feasible for $\las^{\topo}_{t+1}(G)$,
resp.\ $\las'^{,\topo}_{t+1}(G)$, then its restriction is feasible for
$\las^{\topo}_t(G)$, resp.\ $\las'^{,\topo}_{t}(G)$. (Note that every
continuous, positive semidefinite kernel on $I_t \times I_t$ extends
by zero to such a kernel on $I_{t+1} \times I_{t+1}$.)

For the remaining inequality, $\las'^{,\topo}_{t}(G) \le
\las^{\topo}_t(G)$ follows immediately from the definitions, and thus
it suffices to prove that $\las^{\topo}_{2t}(G) \le
\las'^{,\topo}_{t}(G)$. Let $\lambda \in \M_\pm(I_{4t})$ be feasible
for $\las^{\topo}_{2t}$, and consider its restriction to $I_{2t}$.  We
will show that this restriction is feasible for
$\las'^{,\topo}_{t}(G)$.  Since $A_{2t}^* \lambda$ is positive
semidefinite, the restriction of $A_{2t}^* \lambda$ to the diagonal
copy of $I_{2t}$ in $I_{2t} \times I_{2t}$ must be nonnegative.
However, this restriction is just the restriction of $\lambda$ to
$I_{2t}$, which is therefore feasible for $\las'^{,\topo}_{t}(G)$.
\end{proof}

We will use the following convergence property, where $\alpha(G)$ is
the independence number of $G$ (i.e., the size of the largest
packing):

\begin{proposition}[de Laat and Vallentin \cite{de2015semidefinite}]
Let $G$ be a compact topological packing graph. Then
\[ \las^{\topo}_{\alpha(G)}(G) = \las'^{,\topo}_{\alpha(G)}(G) = \alpha(G). \]
\end{proposition}

\begin{remark}
The paper \cite[Theorem~2]{de2015semidefinite} only gives the
statement for $\las'^{,\topo}_{\alpha(G)}(G)$, but their proof extends
word for word to $\las^{\topo}_{\alpha(G)}(G)$.
\end{remark}

The empty set is an isolated point in $I_t$, and for some purposes it
is useful to omit it.  There is an equivalent formulation of $\las'_t$
without $\{ \emptyset \}$ using the Schur complement.  For $\mu$ a
measure on $I_t \setminus \{ \emptyset \}$, let $(A_t^{\ne
  \emptyset})^* \mu$ be the restriction of $A_t^* \mu$ to $I_t
\setminus \{\emptyset\} \times I_t \setminus \{\emptyset\}$, and
define the external tensor product $\mu \otimes \nu$ of two measures
as the measure taking value $\mu(A) \nu(B)$ on $A \times B$.  The
following proposition follows from taking the Schur complement:

\begin{proposition}
The supremum of $\mu(I_{=1})$ over measures $\mu \in \M_+(I_{2t}
\setminus \{ \emptyset \})$ such that $(A_t^{\ne \emptyset})^* \mu -
\mu \otimes \mu$ is positive semidefinite is $\las'^{,\topo}_t(G)$.
\end{proposition}

Here $\mu \otimes \mu$ refers to a measure on $I_t
\setminus\{\emptyset\} \times I_t \setminus\{\emptyset\}$, so that it
is comparable with $(A_t^{\ne \emptyset})^* \mu$.

\begin{proof}
Define a new measure $\widetilde{\mu}$ on $I_{2t}$ that is equal to
$\mu$ on $I_{2t} \setminus \{ \emptyset \}$ and puts a mass of $1$ on
the point $\emptyset$.  The version of the Schur complement in
\cite[Lemma~9]{de2015semidefinite} shows that if $(A_t^{\ne
  \emptyset})^* \mu - \mu \otimes \mu$ is positive semidefinite, then
$A_t^* \widetilde{\mu}$ is a positive semidefinite measure and so
$\widetilde{\mu}$ is a feasible measure for $\las'^{,\topo}_t(G)$.
\end{proof}

\begin{remark} \label{remark:schurcomplement}
One can consider a weakened hierarchy in which the positive
semidefiniteness condition above is weakened to the conditions that
$(A_t^{\ne \emptyset})^* \mu$ is positive semidefinite and that
\[ 2 \mu(I_{=2}) + \mu(I_{=1}) = (A_t^{\ne \emptyset})^* \mu(I_{=1} \times I_{=1})
\ge \mu(I_{=1})^2.\]
This weakening was discussed briefly in \cite[equation~(3) in
Section~2.1]{de2018k}.  In particular, for the case $t = 1$, this
weakening of the hierarchy is equivalent to $\vartheta'^{,\topo}$.
Furthermore, de Laat and Vallentin \cite{de2015semidefinite} show that
$\vartheta'^{,\topo}$ is actually equivalent to $\las_1'^{,\topo}$.
\end{remark}

The convex dual of the optimization problem defining
$\las'^{,\topo}_t(G)$ is as follows:

\begin{definition}
For a compact topological packing graph $G$, we define $\las'^{,
  \topo}_t(G)^*$ to be the infimum of $K(\emptyset, \emptyset)$ over
continuous kernels $K \colon I_t \times I_t \to \R$ such that
\begin{enumerate}
    \item $K$ is positive semidefinite,
    \item $A_t K(\{ x \}) \le - 1$, and
    \item $A_t K(S) \le 0$ for $S \in I_t$ with $|S| \ge 2$.
\end{enumerate}
\end{definition}

By \cite[Theorem~1]{de2015semidefinite}, strong duality holds:

\begin{theorem}[Strong duality] \label{lasstrongduality}
Let $G$ be a compact topological packing graph. Then
\[ \las'^{, \topo}_t(G) = \las'^{, \topo}_t(G)^*. \]
\end{theorem}

Note that the proof of this theorem in the published version of
\cite{de2015semidefinite} contains a minor gap, which is filled in the
arXiv version.

\subsection{The Lasserre hierarchy as sandwich functions}

\begin{definition}
For a finite graph $H$ and for $A = \las_t$ or $\las'_t$, define
$\las_t(G) = \las_t^{\topo}(G)$ and $\las_t'(G) =
\las_t'^{,\topo}(G)$.

For a discrete graph $G$, possibly infinite but with finite clique
covering number, we define $\las_t(G)$, resp.\ $\las'_t(G)$, to be the
supremum over all induced finite subgraphs $H \subseteq G$ of
$\las_t(H)$, resp.\ $\las'_t(H)$.
\end{definition}

This definition coincides on finite graphs with
$\las'^{,\topo}_t(G)$. We do not know whether they agree on all
compact topological packing graphs, even for the case $t = 1$ (which
is $\vartheta'$).  Using Theorem~\ref{comparisontopnon} below, for
compact subsets of $\R^n$, this question can be reformulated purely in
terms of $\las'$.

\begin{theorem}\label{lasserresandwich}
The functions $G \mapsto \las_t(\overline{G})$ and $G \mapsto
\las'_t(\overline{G})$ are sandwich functions.
\end{theorem}

\begin{proof}[Proof of Theorem~\ref{lasserresandwich}]
Recall that Lemma~\ref{finitecheck} shows that we only need to check
the statement on finite graphs.  The case when $G$ is a point is
clear.

Consider a graph homomorphism from $\overline{G}$ to $\overline{H}$,
where $G$ and $H$ are finite graphs, and let $f \colon V(G) \to V(H)$
be the underlying map on vertex sets.  Let us also use $f$ for the
induced map on independent sets $f \colon I_{=k,G} \to I_{=k,H}$ for
each $k$.  If $\mu$ is a measure on $I_{t,G}$, we define the measure
$f_* \mu$ on $I_{t,H}$ via
\[ f_* \mu(A) = \mu(f^{-1}(A)). \]

First, we show that if $\mu$ is feasible for $\las_t(G)$, then $f_*
\mu$ is feasible for $\las_t(H)$ with the same objective function
value.  If $A_t^* \mu$ is positive semidefinite on $I_{t,G} \times
I_{t,G}$, then $A_t^* f_* \mu$ is positive semidefinite on $I_{t,H}
\times I_{t,H}$.  To see why, we will use \eqref{eq:charac}; for
finite graphs, it says that
\[
A_t^*\mu(\{J\}\times\{J'\}) = \begin{cases} \mu(\{J \cup J'\}) &
  \text{if $J \cup J' \in I_{2t,G}$, and}\\ 0 & \text{otherwise.}
\end{cases}
\]
Let $w_J$ be a set of real weights for $J \in I_{t,H}$.  If we set
$w_K = w_J$ for $K \in I_{t,G}$ with $f(K)=J$ and $w_K=0$ otherwise,
then
\[ \begin{split}
\sum_{\substack{J,J' \in I_{t,H}\\J \cup J' \in I_{2t,H}}} w_J w_{J'}&
f_* \mu(\{ J \cup J' \})\\ &= \sum_{\substack{J,J' \in I_{t,H}\\J \cup
    J' \in I_{2t,H}}} w_J w_{J'} \mu(\{ K \cup K' : K \in f^{-1}(J),
K' \in f^{-1}(J')\})\\ &= \sum_{\substack{K,K' \in I_{t,G}\\K \cup K'
    \in I_{2t,G}}} w_K w_{K'} \mu(\{K \cup K'\}) \ge 0.
\end{split}\]
This shows that $\las_t(G) \le \las_t(H)$.  Furthermore, if $\mu$ is
nonnegative, then $f_* \mu$ is nonnegative, and thus $\las'_t(G) \le
\las'_t(H)$ as well.

All the remains to check is the union axiom.  Let $G$ and $H$ be
disjoint graphs, and consider $G\sqcup H$.  We first want to show that
$\las_t(G \sqcup H) \ge \las_t(G) + \las_t(H)$ and similarly for
$\las'_t(G \sqcup H)$.  Let $\mu_G$ and $\mu_H$ be measures on $I_{2t,
  G}$ and $I_{2t, H}$.  To define a measure $\mu$ on $I_{2t, G \sqcup
  H}$, it suffices to define the measure of a singleton $\{ S \}$ with
$S \in I_{2t,G \sqcup H}$.  We define this by
\[\mu(\{ S \}) = \mu_G(\{S \cap V(G)\}) \, \mu_{H}(\{S \cap V(H)\}).\]
If $\mu_G(\{ \emptyset \}) = \mu_H(\{ \emptyset \}) = 1$, then $\mu(\{
\emptyset \}) = 1$ immediately and $\mu(I_{=1}) = \mu_G(I_{=1, G}) +
\mu_{H}(I_{=1, H})$.  Thus, the objective is additive.  Moreover, if
$\mu_G$ and $\mu_H$ are both nonnegative, then $\mu$ is as well.  It
suffices to show that if $A_t^*\mu_G$ and $A_t^*\mu_H$ are positive
semidefinite, then $A_t^*\mu$ is. Because $G \sqcup H$ is finite and
the cone of positive semidefinite matrices is self-dual, $A_t^*\mu$ is
positive semidefinite if and only if the matrix
\[
\big(\mu(\{J \cup J'\})\big)_{J,J' \in I_{t,G \sqcup H}}
\]
is positive semidefinite, where we set $\mu(\{S\})=0$ if $S \not\in
I_{2t,G\sqcup H}$. This matrix is the tensor product of the
corresponding matrices for $\mu_G$ and $\mu_H$, and the desired
conclusion follows.

Finally, we want to show that $\las_t(G \cup H) \le \las_t(G) +
\las_t(H)$ and similarly for $\las'_t$.  Let $\mu$ be a measure on
$I_{2t, G \cup H}$ and consider its restriction $\mu_G$ to $I_{2t, G}$
and restriction $\mu_H$ to $I_{2t, H}$.  We see that
\[ \mu_G(I_{=1, G}) + \mu_H(I_{=1, H}) = \mu(I_{=1, G \cup H}).\]
If $\mu$ is nonnegative, then $\mu_G$ and $\mu_H$ are, and if
$A_t^*\mu$ is positive semidefinite, then $A_t^*\mu_G$ and
$A_t^*\mu_H$ are (they correspond to submatrices of the positive
semidefinite matrix for $G \sqcup H$).  Finally,
\[ \mu(\{ \emptyset \}) = \mu_G(\{ \emptyset \}) = \mu_H(\{ \emptyset \}). \]
By starting with a feasible point for $\las_t(G \sqcup H)$ that is
arbitrarily close to the optimum, we conclude that $\las'_t(G \cup H)
\le \las'_t(G) + \las'_t(H)$, and similarly $\las'_t(G \cup H) \le
\las'_t(G) + \las'_t(H)$.
\end{proof}

Because these functions are packing bounds functions, we can conclude
that a Euclidean limit must exist.  In the next section, we formulate
the Euclidean limit of $\las_t'$ and $\las_t'^{,\topo}$ in terms of a
semidefinite program.

\subsection{The Euclidean limit of the Lasserre hierarchy}
\label{euclideanhierarchysection}

For any subset $C \subseteq \R^n$, consider the graph with vertex set
$C$ and edges consisting of pairs $(x,y)$ with $0 < d(x,y) < 2$. If
$C$ is bounded, we can obtain $\las_t'(C)$ by treating the graph as a
discrete graph. If $C$ is compact, then the graph is a compact
topological packing graph and we can obtain $\las_t'^{,\topo}(C)$. For
any $C$, we obtain a topological packing graph, and thus the
collection $I_{t,C}$ of independent sets of size at most $t$ has a
topology, but it will not be compact in general. For example,
$I_{t,\R^n}$ admits an action of $\R^n$ by translation.

Much like the case of $\vartheta'(C)$, we can interpret $\las_t'(C)$
as a supremum over finitely supported measures coming from finite
subgraphs of $C$. From this point of view, the inequality $\las_t'(C)
\le \las_t'^{,\topo}(C)$ is immediate when $C$ is compact. The
following theorem gives a more precise relationship.

\begin{theorem}\label{comparisontopnon}
Let $C$ be a compact subset of $\R^d$, and let
$\overline{C(\varepsilon)}$ be the closure of $C(\varepsilon)$, the
$\varepsilon$-neighborhood of $C$. Then
\[\las_t'^{,\topo}(C) = \lim_{\varepsilon \to 0^+} \las'_t(C(\varepsilon))
= \lim_{\varepsilon \to 0^+} \las_t'^{,\topo}(\overline{C(\varepsilon)}).\]
\end{theorem}

\begin{proof}
First, we will show that
\[\lim_{\varepsilon \to 0^+} \las'^{,\topo}_t(\overline{C(\varepsilon)})
= \las'^{,\topo}_t(C).\]
We will prove both inequalities separately.  One inequality is
immediate by inclusion of $C \subseteq C(\varepsilon)$.  For the other
direction, since $\overline{C(\varepsilon)}$ is compact, there is a
measure achieving the optimal value of
$\las'^{,\topo}_t(\overline{C(\varepsilon)})$ for any $\varepsilon$.
Take $\mu_n$ achieving this for
$\las'^{,\topo}_t(\overline{C(\varepsilon_n)})$ for a sequence
$\varepsilon_n \to 0$.  This is a sequence of weak-$*$-bounded
measures, and by the Banach-Alaoglu theorem there is a convergent
subsequence in the weak-$*$ topology, which we denote $\mu_n \to \mu$.
We claim that $\mu$ is feasible for $\las'^{,\topo}_t(C)$.  The
support of $\mu$ must be contained in the intersection, and it must be
supported on independent sets by the corresponding property for
$\mu_n$.  Moreover, it must be positive since $\mu_n$ are positive
measures.  To prove positive semidefiniteness for $A_t^*\mu$, it
suffices to check that $\int f(x) f(y) \,dA_t^*\mu(x,y) \ge 0$ for
every continuous $f \colon I_{t,C} \to \R$, because Mercer's theorem
\cite[Theorem~3.11.9]{SimonVol4} allows us to write any positive
semidefinite kernel as an infinite sum of terms of the form $f(x)
f(y)$.  Let $f_\varepsilon$ be a compactly supported, continuous
extension of $f$ to $I_{t,C(\varepsilon)}$.  Using $f_\varepsilon$, we
can write the integral as a limit:
\[\int f(x) f(y) \,dA_t^*\mu(x,y) = \lim_{n\to\infty} \int
f_\varepsilon(x) f_\varepsilon(y) \,dA_t^*\mu_n(x,y).\]
Therefore $A_t^*\mu$ is positive semidefinite, and $\mu$ is feasible
for $\las'^{,\topo}(C)$.  Since the objective is a continuous
functional for the weak-$*$ topology, $\mu$ gives a lower bound of
$\lim_{\varepsilon\to0^+} \las'^{,\topo}_t(\overline{C(\varepsilon)})$
for $\las'^{,\topo}_t(C)$.

As a consequence of this argument, note that for compact sets $C$,
\[\lim_{\varepsilon \to 0^+} \las_t'^{,\topo}((1+\varepsilon)C) = \las_t'^{,\topo}(C),\]
because for every $\varepsilon>0$, there is some $\varepsilon'>0$ such
that $(1+\varepsilon')C$ embeds into $\overline{C(\varepsilon)}$.

Now, to complete the proof it suffices to prove
that \[\las'^{,\topo}_t(C) \le \lim_{\varepsilon\to0^+}
\las'_t(C(\varepsilon)).\] Let $Y$ be a geometric simplicial complex
such that $C \subseteq Y \subseteq C(\varepsilon)$ and every simplex
has diameter at most $\varepsilon$, constructed as in
Lemma~\ref{geometricsimplicial}.  We use $Y_0$ to denote the vertices
of $Y$ and $Y_d$ to denote its top-dimensional simplices.  For any
positive semidefinite kernel on $I_{t, Y_0}$, we will produce a
positive semidefinite kernel on the whole $I_{t, Y}$ by viewing the
simplicial complex as a union of finite elements and by taking convex
combinations and extending the kernel by linearity.  This is the same
idea as in the proof of Proposition~\ref{thetaduality}, but with more
cumbersome notation because of the use of $I_t$. Using this approach,
we will obtain feasible kernels for $\las_t'^{,\topo}(C)^*$ from those
for $\las_t'(C(\varepsilon))^*$, specifically those using the subset
$Y_0$.

Let $K$ be a positive semidefinite kernel on $I_{t,Y_0}$.  Then we can
define a positive semidefinite kernel $L$ on $I_{t,Y}$ as
follows. Given any point in $Y$, we can randomly round it to a vertex
in $Y_0$ by using barycentric coordinates: if the point is $y =
\lambda_0 y_0 + \dots + \lambda_d y_d$ with $\lambda_i \ge 0$, $\sum_i
\lambda_i = 1$, and $\{y_0,\dots,y_d\}$ a top-dimensional simplex,
then we round $y$ to $y_i$ with probability $\lambda_i$. This process
is well defined, because the only way $y$ can be in several
top-dimensional simplices is if all the weights not coming from their
intersection vanish. Similarly, we can round an independent set $J$ by
rounding each point in it independently.  We will denote the rounded
version of $J$ by the random variable $r(J)$. One subtlety is that
$r(J)$ may not be an independent set, because two points at distance
less than $2+2\varepsilon$ may round to points at distance less than
$2$ (recall that the simplices have diameter at most
$\varepsilon$). That will not be a problem, since we can extend $K$ by
zero to obtain a positive semidefinite kernel on arbitrary sets of
size at most $t$, not just independent sets.

Using this notion of rounding, we define $L(J,J')$ as the expected
value \[\mathbb{E} K(r(J),r(J'))\] of $K(r(J),r(J'))$ when we round
each point in $J \cup J'$ independently. Then $L$ is a continuous
function on $I_{t,Y} \times I_{t,Y}$. To show that it is a positive
semidefinite kernel, we must show that for all weights $w_J \in \R$
for $J \in I_{t,Y}$ that vanish for all but finitely many $J$,
\[
\sum_{J,J' \in I_{t,Y}} w_J w_{J'} L(J,J') \ge 0.
\]
To prove this inequality, consider randomly rounding the points of
$Y$ independently. Then
\[
\sum_{J,J' \in I_{t,Y}} w_J w_{J'} L(J,J') = \mathbb{E} \sum_{J,J' \in
  I_{t,Y}} w_J w_{J'} K(r(J),r(J')),
\]
which is nonnegative because $K$ is a positive semidefinite kernel.

All the remains to check is the conditions on $A_tL$ for a feasible
kernel.  If $A_tK(\{x\}) \le -1$ for all $x \in Y_0$, then
$A_tL(\{x\}) \le -1$ for all $x \in Y$, because $A_tL(\{x\}) =
\mathbb{E} A_tK(r(\{x\}))$. However, the case of $A_tL(S)$ with $|S|
\ge 2$ is slightly more subtle. The issue is that $r(S)$ may not be an
independent set even if $S$ is. However, if the minimal distance
between points in $S$ is at least $2+2\varepsilon$, then $r(S)$ must
always be an independent set. In that case,
\[
A_tL(S) = \sum_{\substack{J,J' \subseteq S\\ J \cup J' = S}} L(J,J') =
\mathbb{E} \sum_{\substack{J,J' \subseteq r(S)\\ J \cup J' = r(S)}}
K(J,J') \le 0
\]
if $A_tK(S) \le 0$. In other words, we obtain a feasible kernel for
$(1+\varepsilon)^{-1}C$, rather than $C$.

After rescaling space by a factor of $1+\varepsilon$, we conclude that
\[\las'^{,\topo}_t(C) \le \las'_t((1+\varepsilon)Y_0) \le
\las'_t((1+\varepsilon) C(\varepsilon)).\]
For any $\varepsilon'$, we can choose $\varepsilon$ so that
$(1+\varepsilon) C(\varepsilon) \subseteq C(\varepsilon')$ and letting
$\varepsilon' \to 0$ proves the inequality, and hence the result.
\end{proof}

\begin{corollary}\label{cubelimlas}
Let $I^n$ be the unit cube in $\R^n$. Then
\[\lim_{r\to\infty} \frac{\las'^{, \topo}_t(rI^n)}{r^n} =
\lim_{r \to\infty} \frac{\las'_t(rI^n)}{r^n}.\]
\end{corollary}

\begin{proof}
This corollary follows from Theorem~\ref{comparisontopnon} together
with the fact that $I^n(\varepsilon)$ embeds into
$(1+2\varepsilon)I^n$.
\end{proof}

We are now ready to give the Euclidean limit of the Lasserre
hierarchy.  Let $\R^n$ act on Radon measures on $I_{2t, \R^n}$ by
translation, and consider translation-invariant measures.  Any
translation-invariant measure $\mu$ will restrict to $I_{=1, \R^n}$ as
some multiple of the Lebesgue measure, and we define
$\widehat{\mu}(0)$ to be this multiple.  That is,
\[ \mu |_{I_{=1, \R^n}}(C) = \widehat{\mu}(0) \mathcal{L}_n(C) \]
for Borel sets $C \subseteq \R^n$.

\begin{definition}
Let $\las'_t(\R^n)$ be the supremum of $\widehat{\mu}(0)$ over all
translation-invariant Radon measures $\mu$ on $I_{2t, \R^n}$ such that
\begin{enumerate}
    \item $A_t^* \mu$ is positive semidefinite as a measure on $I_{t,
      \R^n} \times I_{t, \R^n}$,
    \item $\mu$ is nonnegative, and
    \item $\mu(\{\emptyset\})=1$.
\end{enumerate}
We call such a measure $\mu$ a \emph{correlation measure of order
$2t$} with \emph{center density $\widehat{\mu}(0)$}.
\end{definition}

Let $\P$ be a periodic packing in $\R^n$, i.e., the union of finitely
many translates of a lattice $\Lambda$ such that no two points of $\P$
are closer than distance~$2$ apart, and let $D$ be a fundamental
parallelotope for $\Lambda$. To obtain a correlation measure from
$\P$, we define $\mu_\P$ to be the Radon measure on $I_{2t, \R^n}$
characterized by
\[
\int_{I_{2t, \R^n}} f \, d\mu_\P = \frac{1}{\mathcal{L}_n(D)}\int_{D}
\sum_{\substack{S \subseteq \P+v\\|S| \le 2t}} f(S) \, dv
\]
for compactly supported, continuous functions $f \colon I_{2t,
  \R^n}\to \R$.  In other words, $\mu_\P|_{I_{=k}}$ is essential the
correlation function of order $k$ for $\P$. The purpose of averaging
over $v \in D$ is to make $\mu_\P$ translation-invariant.  The center
density $\widehat{\mu}_\P(0)$ is the usual center density of the
sphere packing $\P$, i.e., $N/\mathcal{L}_n(D)$ if $\P$ consists of
$N$ translates of $\Lambda$.

To show that $\mu_P$ is a correlation measure, all that remains is to
prove that $A_t^*\mu_\P$ is positive semidefinite. Let $K \colon I_t
\times I_t \to \R$ be a continuous, positive semidefinite kernel with
compact support. Then
\begin{align*}
\int_{I_{t, \R^n} \times I_{t, \R^n}} K \, dA_t^*\mu_P &= \int_{I_{2t,
    \R^n}} A_tK \, d\mu_P\\ &= \frac{1}{\mathcal{L}_n(D)}\int_{D}
\sum_{\substack{S \subseteq \P+v\\|S| \le 2t}} \sum_{\substack{J,J'
    \in I_t\\ J \cup J' = S}} K(J,J') \, dv\\ &=
\frac{1}{\mathcal{L}_n(D)}\int_{D} \sum_{\substack{J,J' \in I_t\\ J
    ,J' \subseteq \P+v}} K(J,J') \, dv,
\end{align*}
and
\[
\sum_{\substack{J,J' \in I_t\\ J ,J' \subseteq \P+v}} K(J,J') \ge 0
\]
because $K$ is positive semidefinite. (Note that this is a finite sum,
because $K$ has compact support.)

The main result of this section is that the quantity $\las'_t(\R^n)$
is the Euclidean limit of the packing bound function $\las'_t$.

\begin{theorem} \label{theorem:lasrn}
For each $n$ and $t$,
\[ \las'_t(\R^n) = \delta_{\las'_t, n}. \]
\end{theorem}

\begin{proof}
Because of Corollary~\ref{cubelimlas}, it suffices to check that
\[ \lim_{r \to \infty} \frac{\las'^{,\topo}_t(rI^n)}{r^n} = \las'_t(\R^n). \]
If $\mu$ is a correlation measure of order $2t$ (i.e., feasible for
$\las'_t(\R^n)$), then restricting $\mu$ to $I_{2t, rI^n}$ gives a
feasible measure for $\las'^{,\topo}_t(rI^n)$ with objective $r^n
\las'_t(\R^n)$.  This shows that
\[ \frac{\las'^{,\topo}_t(rI^n)}{r^n} \ge \las'_t(\R^n) \]
for any $r$.

To prove the other direction of the inequality, we use a tiling
construction.  In this case, we use $\las_t'$ instead of
$\las_t'^{,\topo}$. Equivalently, we restrict our attention to
feasible measures with finite support. Given such a measure for
$rI^n$, we can extend it to $\R^n$ as follows.

Consider a tiling $\R^n = \bigcup_{i \in \mathcal{T}} T_i$, where each
tile $T_i$ is congruent to $(r+2)I^n$ and the tiles are translates
under the cubic lattice $(r+2)\Z^n \subseteq \R^n$.  Let $Q_i
\subseteq T_i$ be the centered copy of $rI^n$ in $T_i$ (with $g_i
\colon Q_i \to rI^n$ performing this identification), so the induced
subgraph of $Q := \bigcup_{i \in \mathcal{T}} Q_i$ is a disjoint union
over $i \in \mathcal{T}$.  Let $\mu_{rI^n}$ be a finitely supported
measure on $I_{2t,rI^n}$ such that $\mu_{rI^n}(\{ \emptyset \}) = 1$.
We now define a measure $\mu_Q$ on $I_{2t,Q}$ as follows. For any $S
\in I_{2t,Q}$, let
\[ \mu_{Q}(\{ S \}) = \prod_{i \in \mathcal{T}} \mu_{rI^n}(\{ g_i(S \cap Q_i) \}). \]
This equation directly defines $\mu_Q$ for all single-element subsets
of $I_{2t,Q}$, and for all Borel subsets as an atomic measure. Note in
particular that for each $S$, all but finitely many factors in the
infinite product are $1$. Furthermore, for any compact set $C
\subseteq Q$, the measure $\mu_Q|_{I_{2t,C}}$ has finite support. We
extend $\mu_Q$ to a measure on $I_{2t,\R^n}$ by zero.

Next, we show that $A_t^*\mu_Q$ is positive semidefinite. In other
words,
\[
\iint_{I_{t,\R^n} \times I_{t,\R^n}} K \, dA_t^*\mu_Q \ge 0
\]
for every compactly supported, continuous, positive semidefinite
kernel $K \colon I_{t,\R^n} \times I_{t,\R^n} \to \R$. Every compact
subset of $I_{t,\R^n}$ is contained in $I_{t,C}$ for some compact
subset $C$ of $\R^n$, and thus only finitely many cubes $Q_i$ play a
role for any given $K$. Because $\mu_{rI^n}$ has finite support, what
we need to check is an assertion about positive semidefinite
matrices. Specifically, the relevant matrix for $A_t^* \mu_Q$ is a
tensor power of that for $A_t^* \mu_{rI^n}$, just as in the
verification of the union axiom in the proof of
Theorem~\ref{lasserresandwich}, and positive semidefiniteness is
therefore preserved.

All that remains is to average $\mu_Q$ under the action of $\R^n$ by
translation.  Because $\mu$ is already invariant under translation by
a lattice $(r+2) \Z^n$, we can average over the action of the quotient
torus, which is a compact group.  Therefore, the average is well
defined.

By construction, if $\mu_{rI^n}$ is feasible for $\las'_t(rI^n)$, then
the result $\mu$ after averaging is a translation-invariant measure on
$I_{2t, \R^n}$ that is feasible for $\las'_t(\R^n)$, with objective
\[ \frac{\mu_{rI^n}(I_{=1,rI^n})}{(r+2)^n}. \]
Letting $r$ be arbitrarily large and optimizing over all choices of
$\mu_{rI^n}$ gives the result.
\end{proof}

\begin{corollary}
For each $t$, $\las'_t(\R^n)$ is an upper bound on sphere packing, and
\[ \lim_{t\to\infty} \las'_t(\R^n) = \delta_{\pack,n}. \]
\end{corollary}

\begin{proof}
For fixed $r$, we may choose $t$ sufficiently large so that
$\las'_t(rI^n) = \alpha(rI^n)$ and $r^{-n} \alpha(rI^n)$ gives an
upper bound for $\las'_t(\R^n)$.  As $r$ becomes large, this bound
will come arbitrarily close to the optimal sphere center density.
\end{proof}

We can formulate an optimization problem dual to $\las'_t(\R^n)$ as
follows.

\begin{definition}
Let $\las'_t(\R^n)^*$ be the infimum of $K(\emptyset,\emptyset)$ over
all continuous kernels $K \colon I_{t,\R^n} \times I_{t,\R^n} \to \R$
with compact support such that
\begin{enumerate}
    \item $K$ is positive semidefinite,
    \item $A_tK(S) \le 0$ whenever $|S| > 1$, and
    \item $\int_{I_{=1,\R^n}} A_tK \, d\mathcal{L}_n \le -1$,
\end{enumerate}
where we view Lebesgue measure $\mathcal{L}_n$ as a measure on
$I_{=1,\R^n}$ by identifying $I_{=1,\R^n}$ with $\R^n$. We call such a
$K$ an \emph{auxiliary function of order $2t$}.
\end{definition}

\begin{remark}
One unsatisfying feature of the above optimization problem is that the
kernel $K$ cannot be made invariant under the action of $\R^n$, as
that would require that $K(\emptyset, \{ x \})$ take a constant value,
contradicting the third condition in the definition.  To formulate a
dual problem in a way that allows for solutions invariant under the
group action, one could use the Schur complement formulation discussed
in Remark~\ref{remark:schurcomplement}, at the cost of complicating
the statement of the optimization problem.
\end{remark}

Let $\mu$ be a correlation measure and $K$ an auxiliary function, both
of order $2t$. Weak duality follows immediately from
\[
0 \le \int_{I_t \times I_t} K \, dA_t^*\mu = \int_{I_{2t}} A_tK \,
d\mu \le K(\emptyset,\emptyset) - \widehat{\mu}(0).
\]
The relationship with $\las'^{,\topo}_t$ is simple. Given any
correlation measure $\mu$ of order $2t$ for $\R^n$, restricting $\mu$
to $I_{2t,rI^n}$ gives a feasible measure for $\las_t'^{,\topo}(rI^n)$
with objective $r^n \widehat{\mu}(0)$, and letting $r \to \infty$
shows that
\[
\las_t'(\R^n) \le \delta_{\las'_t,n}
\]
by Corollary~\ref{cubelimlas} (this argument is the first part of the
proof of Theorem~\ref{theorem:lasrn}). Conversely, suppose $K$ is any
feasible kernel for $\las_t'^{,\topo}(rI^n)^*$. Then extending $r^{-n}
K$ by zero gives a feasible kernel for $\las_t'(\R^n)^*$ with
objective $r^{-n} K(\emptyset,\emptyset)$, and thus
\[
\las_t'(\R^n)^* \le \delta_{\las'_t,n}.
\]
By combining these inequalities with Theorem~\ref{theorem:lasrn}, we
obtain strong duality:

\begin{theorem}
For each $n$ and $t$,
\[ \las'_t(\R^n) = \las'_t(\R^n)^* = \delta_{\las_t',n}. \]
\end{theorem}

Note that the optimum in $\las'_t(\R^n)^*$ will generally not be
achieved unless we broaden the class of auxiliary functions.

\begin{remark}
The Lasserre hierarchy bound $\las_1'(\R^n)$ is equivalent to the
linear programming bound, because
$\las_1'^{,\topo}=\vartheta'^{,\topo}$, as shown in
\cite[Theorem~3]{de2015semidefinite}. It is therefore sharp for $n =
1$, $8$, and $24$, and conjecturally for $n = 2$ (see
\cite{cohn2003new,viazovska2017,CKMRV2017}).  For which other pairs
$(n,t)$ might there be a sharp bound? It is unclear whether these
sharp bounds are a peculiar phenomenon for $t=1$, or whether we can
expect further cases with $t>1$. It is even conceivable that for each
dimension $n$, some finite value of $t$ yields a sharp bound.
\end{remark}

One hint that additional sharp bounds might be possible in Euclidean
space comes from the case of binary codes of block length~$20$ and
minimal distance~$8$. Gijswijt, Mittelmann, and Schrijver \cite{GMS}
obtained a sharp bound for the size of such a code using $\las_2'$
(see Section~VIII of their paper for the reduction to $\las_2'$).

\section*{Acknowledgements}

This paper is a spinoff from a larger project with David de Laat, with
whom we had numerous helpful discussions but who declined
coauthorship.  We also thank Austin Anderson, Alexander Reznikov,
Oleksandr Vlasiuk, and Edward White for pointing out reference
\cite{kolmogorovtikhomirov}.


\begin{thebibliography}{10}

\bibitem{alon} N.\ Alon, \emph{The {S}hannon capacity of a union},
  Combinatorica \textbf{18} (1998), no.~3, 301--310. \MR{1721946}
  \doi{10.1007/PL00009824}

\bibitem{ambrosiofuscopallara} L.\ Ambrosio, N.\ Fusco, and
  D.\ Pallara, \emph{Functions of bounded variation and free
  discontinuity problems}, Oxford Mathematical Monographs, The
  Clarendon Press, Oxford University Press, New York,
  2000.\ \MR{1857292}

\bibitem{ARVW} A.\ Anderson, A.\ Reznikov, O.\ Vlasiuk, and E.\ White,
  \emph{Polarization and covering on sets of low smoothness}, in
  preparation, 2021.

\bibitem{bachoc2012invariant} C.\ Bachoc, D.~C.\ Gijswijt,
  A.\ Schrijver, and F.\ Vallentin, \emph{Invariant semidefinite
  programs}, Handbook on semidefinite, conic and polynomial
  optimization (M.~F.\ Anjos and J.~B.\ Lasserre, eds.),
  Internat.\ Ser.\ Oper.\ Res.\ Management Sci., vol.\ 166, Springer,
  New York, 2012, pp.~219--269.  \arXiv{1007.2905} \MR{2894697}
  \doi{10.1007/978-1-4614-0769-0_9}

\bibitem{bachoc2008new} C.\ Bachoc and F.\ Vallentin, \emph{New upper
bounds for kissing numbers from semidefinite programming},
  J.\ Amer.\ Math.\ Soc.\ \textbf{21} (2008), no.~3,
  909--924. \arXiv{math/0608426} \MR{2393433}
  \doi{10.1090/S0894-0347-07-00589-9}

\bibitem{ghanem1998explicit} R.\ Ben~Ghanem and C.\ Frappier,
  \emph{Explicit quadrature formulae for entire functions of
  exponential type}, J.\ Approx.\ Theory \textbf{92} (1998), no.~2,
  267--279. \MR{1604935} \doi{10.1006/jath.1997.3122}

\bibitem{borodachov2007asymptotics} S.~V.\ Borodachov, D.~P.\ Hardin,
  and E.~B.\ Saff, \emph{Asymptotics of best-packing on rectifiable
  sets}, Proc.\ Amer.\ Math.\ Soc.\ \textbf{135} (2007), no.~8,
  2369--2380. \arXiv{math-ph/0605021} \MR{2302558}
  \doi{10.1090/S0002-9939-07-08975-7}

\bibitem{borodachov2019discrete} \bysame, \emph{Discrete energy on
rectifiable sets}, Springer Monographs in Mathematics, Springer, New
  York, 2019. \MR{3970999} \doi{10.1007/978-0-387-84808-2}

\bibitem{brassmoserpach} P.\ Brass, W.\ Moser, and J.\ Pach,
  \emph{Research problems in discrete geometry}, Springer, New York,
  2005. \MR{2163782} \doi{10.1007/0-387-29929-7}

\bibitem{cohnmeasuretheory} D.~L.\ Cohn, \emph{Measure theory}, second
  ed., Birkh\"{a}user Advanced Texts: Basler Lehrb\"{u}cher,
  Birkh\"{a}user/Springer, New York, 2013. \MR{3098996}
  \doi{10.1007/978-1-4614-6956-8}

\bibitem{cohn2002new} H.\ Cohn, \emph{New upper bounds on sphere
packings {II}}, Geom.\ Topol.\ \textbf{6} (2002),
  329--353. \arXiv{math.MG/0110010} \MR{1914571}
  \doi{10.2140/gt.2002.6.329}

\bibitem{cohn2018gaussian} H.\ Cohn and M.\ de~Courcy-Ireland,
  \emph{The {G}aussian core model in high dimensions}, Duke
  Math.\ J.\ \textbf{167} (2018), no.~13, 2417--2455.
  \arXiv{1603.09684} \MR{3855354} \doi{10.1215/00127094-2018-0018}

\bibitem{cohn2003new} H.\ Cohn and N.\ Elkies, \emph{New upper bounds
on sphere packings {I}}, Ann.\ of Math.\ (2) \textbf{157} (2003),
  no.~2, 689--714. \arXiv{math.MG/0110009} \MR{1973059}
  \doi{10.4007/annals.2003.157.689}

\bibitem{CKMRV2017} H.\ Cohn, A.\ Kumar, S.~D.\ Miller, D.\ Radchenko,
  and M.\ Viazovska, \emph{The sphere packing problem in dimension
  24}, Ann.\ of Math.\ (2) \textbf{185} (2017), no.~3,
  1017--1033. \arXiv{1603.06518} \MR{3664817}
  \doi{10.4007/annals.2017.185.3.8}

\bibitem{threepointbounds2021} H.\ Cohn, D.\ de~Laat, and A.\ Salmon,
  \emph{Three-point bounds for sphere packing}, in preparation, 2021.

\bibitem{cohn2014sphere} H.\ Cohn and Y.\ Zhao, \emph{Sphere packing
bounds via spherical codes}, Duke Math.\ J.\ \textbf{163} (2014),
  no.~10, 1965--2002. \arXiv{1212.5966} \MR{3229046}
  \doi{10.1215/00127094-2738857}

\bibitem{dymmckean} H.\ Dym and H.~P.\ McKean, \emph{Fourier series
and integrals}, Academic Press, New York-London, 1972. \MR{0442564}

\bibitem{federer2014geometric} H.\ Federer, \emph{Geometric measure
theory}, reprint of the 1969 ed., Classics in Mathematics,
  Springer-Verlag, 1996. \MR{0257325} \doi{10.1007/978-3-642-62010-2}

\bibitem{FejesToth} L.\ Fejes~T\'{o}th, \emph{Regular figures}, A
  Pergamon Press Book, The Macmillan Co., New York, 1964. \MR{0165423}

\bibitem{GMS} D.~C.\ Gijswijt, H.~D.\ Mittelmann, and A.\ Schrijver,
  \emph{Semidefinite code bounds based on quadruple distances}, IEEE
  Trans.\ Inform.\ Theory \textbf{58} (2012), no.~5,
  2697--2705. \arXiv{1005.4959} \MR{2952510}
  \doi{10.1109/TIT.2012.2184845}

\bibitem{gorbachev2000extremum} D.~V.\ Gorbachev, \emph{Extremum
problems for entire functions of exponential spherical type},
  Math.\ Notes \textbf{68} (2000), no.~2, 159--166.  \MR{1822646}
  \doi{10.1007/BF02675341}

\bibitem{gorbachev2001extremum} \bysame, \emph{Extremum problem for
periodic functions supported in a ball}, Math.\ Notes \textbf{69}
  (2001), no.~3, 313--319. \MR{1846833} \doi{10.1023/A:1010275206760}

\bibitem{Groemer} H.\ Groemer, \emph{Existenzs\"{a}tze f\"{u}r
{L}agerungen im {E}uklidischen {R}aum}, Math.\ Z.\ \textbf{81} (1963),
  260--278. \MR{163222} \doi{10.1007/BF01111546}

\bibitem{grozev1995quadrature} G.~R.\ Grozev and Q.~I.\ Rahman,
  \emph{A quadrature formula with zeros of {B}essel functions as
  nodes}, Math.\ Comp.\ \textbf{64} (1995), no.~210,
  715--725. \MR{1277767} \doi{10.2307/2153447}

\bibitem{Hales} T.~C.\ Hales, \emph{A proof of the {K}epler
conjecture}, Ann.\ of Math.\ (2) \textbf{162} (2005), no.~3,
  1065--1185. \MR{2179728} \doi{10.4007/annals.2005.162.1065}

\bibitem{HalesMcLaughlin} T.~C.\ Hales and S.\ McLaughlin, \emph{The
dodecahedral conjecture}, J.\ Amer.\ Math.\ Soc.\ \textbf{23} (2010),
  no.~2, 299--344. \arXiv{math/9811079} \MR{2601036}
  \doi{10.1090/S0894-0347-09-00647-X}

\bibitem{Flyspeck} T.\ Hales, M.\ Adams, G.\ Bauer, T.~D.\ Dang,
  J.\ Harrison, L.~T.\ Hoang, C.\ Kaliszyk, V.\ Magron,
  S.\ McLaughlin, T.~T.\ Nguyen, Q.~T.\ Nguyen, T.\ Nipkow, S.\ Obua,
  J.\ Pleso, J.\ Rute, A.\ Solovyev, T.~H.~A.\ Ta, N.~T.\ Tran,
  T.~D.\ Trieu, J.\ Urban, K.\ Vu, and R.\ Zumkeller, \emph{A formal
  proof of the {K}epler conjecture}, Forum Math.\ Pi \textbf{5}
  (2017), e2, 29.  \arXiv{1501.02155} \MR{3659768}
  \doi{10.1017/fmp.2017.1}

\bibitem{hardin2005minimal} D.~P.\ Hardin and E.~B.\ Saff,
  \emph{Minimal {R}iesz energy point configurations for rectifiable
  {$d$}-dimensional manifolds}, Adv.\ Math.\ \textbf{193} (2005),
  no.~1, 174--204. \arXiv{math-ph/0311024} \MR{2132763}
  \doi{10.1016/j.aim.2004.05.006}

\bibitem{HSV} D.\ Hardin, E.~B.\ Saff, and O.\ Vlasiuk,
  \emph{Asymptotic properties of short-range interaction functionals},
  preprint, 2020. \arXiv{2010.11937}

\bibitem{HJ} R.~A.\ Horn and C.~R.\ Johnson, \emph{Matrix analysis},
  second ed., Cambridge University Press, Cambridge,
  2013. \MR{2978290}

\bibitem{knuth} D.~E.\ Knuth, \emph{The sandwich theorem},
  Electron.\ J.\ Combin.\ \textbf{1} (1994), Article 1, 49
  pp. \arXiv{math/9312214} \MR{1269161} \doi{10.37236/1193}

\bibitem{kolmogorovselected} A.~N.\ Kolmogorov, \emph{Selected
works. {III}. {I}nformation theory and the theory of algorithms},
  Springer Collected Works in Mathematics, Springer, Dordrecht,
  2019. \MR{3822138} \doi{10.1007/978-94-017-2973-4}

\bibitem{kolmogorovtikhomirov} A.~N.\ Kolmogorov and
  V.~M.\ Tikhomirov, \emph{{$\varepsilon$}-entropy and
  {$\varepsilon$}-capacity of sets in function spaces}, Uspehi
  Mat.\ Nauk \textbf{14} (1959), no.~2 (86), 3--86. \MR{0112032}

\bibitem{de2018k} D.\ de~Laat, F.~C.\ Machado,
  F.~M.\ de~Oliveira~Filho, and F.\ Vallentin, \emph{$k$-point
  semidefinite programming bounds for equiangular lines},
  Math.\ Program., to appear. \arXiv{1812.06045}
  \doi{10.1007/s10107-021-01638-x}

\bibitem{de2015semidefinite} D.\ de~Laat and F.\ Vallentin, \emph{A
semidefinite programming hierarchy for packing problems in discrete
geometry}, Math.\ Program.\ \textbf{151} (2015), no.~2, Ser.\ B,
  529--553. \arXiv{1311.3789} \MR{3348162}
  \doi{10.1007/s10107-014-0843-4}

\bibitem{Lasserre} J.~B.\ Lasserre, \emph{An explicit equivalent
positive semidefinite program for nonlinear {$0$}-{$1$} programs},
  SIAM J.\ Optim.\ \textbf{12} (2002), no.~3, 756--769. \MR{1884916}
  \doi{10.1137/S1052623400380079}

\bibitem{Laurent} M.\ Laurent, \emph{A comparison of the
{S}herali-{A}dams, {L}ov\'{a}sz-{S}chrijver, and {L}asserre
relaxations for 0-1 programming}, Math.\ Oper.\ Res.\ \textbf{28}
  (2003), no.~3, 470--496. \MR{1997246}
  \doi{10.1287/moor.28.3.470.16391}

\bibitem{LaurentRendl} M.\ Laurent and F.\ Rendl, \emph{Semidefinite
programming and integer programming}, Discrete optimization
  (K.\ Aardal, G.\ Nemhauser, and R.\ Weismantel, eds.), Handbooks in
  Operations Research and Management Science, vol.~12, Elsevier, 2005,
  pp.~393--514. \doi{10.1016/S0927-0507(05)12008-8}

\bibitem{mattila1995geometric} P.\ Mattila, \emph{Geometry of sets and
measures in {E}uclidean spaces: fractals and rectifiability},
  Cambridge Studies in Advanced Mathematics, vol.~44, Cambridge
  University Press, Cambridge, 1995. \MR{1333890}
  \doi{10.1017/CBO9780511623813}

\bibitem{SadocMosseri} J.-F.\ Sadoc and R.\ Mosseri, \emph{Geometrical
frustration}, Collection Al\'{e}a-Saclay: Monographs and Texts in
  Statistical Physics, Cambridge University Press, Cambridge,
  1999. \MR{1716082} \doi{10.1017/CBO9780511599934}

\bibitem{siegel} C.~L.\ Siegel, \emph{\"{U}ber {G}itterpunkte in
convexen {K}\"{o}rpern und ein damit zusammenh\"{a}ngendes
{E}xtremalproblem}, Acta Math.\ \textbf{65} (1935), no.~1,
  307--323. \MR{1555407} \doi{10.1007/BF02420949}

\bibitem{SimonVol4} B.\ Simon, \emph{Operator theory}, A Comprehensive
  Course in Analysis, Part 4, American Mathematical Society,
  Providence, RI, 2015. \MR{3364494} \doi{10.1090/simon/004}

\bibitem{SteinWeiss} E.~M.\ Stein and G.\ Weiss, \emph{Introduction to
{F}ourier analysis on {E}uclidean spaces}, Princeton Mathematical
  Series, No.\ 32, Princeton University Press, Princeton, NJ,
  1971. \MR{0304972}

\bibitem{viazovska2017} M.~S.\ Viazovska, \emph{The sphere packing
problem in dimension 8}, Ann.\ of Math.\ (2) \textbf{185} (2017),
  no.~3, 991--1015. \arXiv{1603.04246} \MR{3664816}
  \doi{10.4007/annals.2017.185.3.7}

\end{thebibliography}
\end{document}